\documentclass[leqno,11pt]{amsart}

\usepackage{graphicx}
\usepackage[english]{babel}
\usepackage[T1]{fontenc}
\usepackage[utf8]{inputenc}
\usepackage{epsfig}
\usepackage{fancyhdr}
\usepackage{mathrsfs}
\usepackage{amssymb}
\usepackage{amsmath}
\usepackage{amsfonts}
\usepackage{amsthm}
\usepackage{amscd}
\usepackage{color} 

\usepackage{hyperref}
\usepackage[active]{srcltx}

\hypersetup{
  pdftitle={On the equivalence of stochastic completeness, Liouville and Khas'minskii condition in linear and nonlinear setting},
  pdfauthor={Luciano Mari, Daniele Valtorta},
  pdfsubject={Khas'minskii condition for particular operators},
  pdfkeywords={}
  pdfpagelayout=SinglePage,
  pdfpagemode=UseOutlines}

\newcommand{\R}{\mathbb{R}}

\newcommand{\norm}[1]{\left\|#1\right\|}
\newcommand{\ps}[2]{\left\langle#1\middle\vert#2\right\rangle}

\newcommand{\ton}[1]{\left(#1\right)}
\newcommand{\abs}[1]{\left|#1\right|}
\newcommand{\loc}{\mathrm{loc}}
\newcommand{\ra}{\rightarrow}
\newcommand{\oF}{\mathcal{F}}
\newcommand{\oA}{\mathcal{A}}
\newcommand{\oB}{\mathcal{B}}
\newcommand{\diver}{\mathrm{div}}
\newcommand{\Se}{\mathscr{S}}
\newcommand{\eps}{\varepsilon}
\newcommand{\esssup}{\mathrm{esssup}}
\newcommand{\essinf}{\mathrm{essinf}}
\newcommand{\lip}{\mathrm{Lip}}
\newcommand{\disp}{\displaystyle}
\newcommand{\hol}{\mathrm{H\ddot{o}l}}
\newcommand{\pa}{\mathrm{pa}}
\newcommand{\stc}{\mathrm{st}}
\newcommand{\di}{\mathrm{d}}
\newcommand{\vol}{\mathrm{vol}}

\newcommand{\Ka}{\text{Khas'minskii }}

\newcommand{\K}{\mathcal K}
\newcommand{\Kpt}{\mathcal K_{\psi,\theta}}
\newcommand{\A}{\mathcal{A}}

\newcommand{\F}{\mathcal{F}}

\newcommand{\cc}{C^{\infty}_c(\Omega)}
\newcommand{\wup}{W^{1,p}(\Omega)}
\newcommand{\wupst}{W^{1,p}(\Omega)^\star}
\newcommand{\wupz}{W^{1,p}_0(\Omega)}

\newtheorem{prop}{Proposition}[section]
\newtheorem{teo}[prop]{Theorem}
\newtheorem{deph}[prop]{Definition}
\newtheorem{lemma}[prop]{Lemma}
\newtheorem{oss}[prop]{Remark}
\newtheorem{ex}[prop]{Example}
\newtheorem{rem}[prop]{Remark}
\newtheorem{cor}[prop]{Corollary}
\newtheorem*{notaz}{Notational conventions}

\begin{document}

\title[Liouville and Khas'minskii condition]{On the equivalence of stochastic completeness, Liouville and Khas'minskii condition in linear and nonlinear setting}

\date{\today}

\author{Luciano Mari}
\address{Luciano Mari\\ Dipartimento di Matematica\\
Universit\`a degli studi di Milano\\
via Saldini 50\\
20133 Milano, Italy, EU}
\email{luciano.mari@unimi.it}

\author{Daniele Valtorta}
\address{Daniele Valtorta \\ Dipartimento di Matematica\\
Universit\`a degli studi di Milano\\
via Saldini 50\\
20133 Milano, Italy, EU}
\email{danielevaltorta@gmail.com}



\begin{abstract}
Set in Riemannian enviroment, the aim of this paper is to present and discuss some equivalent characterizations of the Liouville property relative to special operators, in some sense modeled after the $p$-Laplacian with potential. In particular, we discuss the equivalence between the Lioville property and the \Ka condition, i.e. the existence of an exhaustion functions which is also a supersolution for the operator outside a compact set. This generalizes a previous result obtained by one of the authors and answers to a question in \cite{PRS1}.
\end{abstract}

\maketitle

\section{Introduction}
In what follows, let $M$ denote a connected Riemannian manifold of
dimension $m$, with no boundary. We stress that no completeness
assumption is required. The relationship between the probabilistic
notions of stochastic completeness and parabolicity (respectively the
non-explosion and the recurrence of the Brownian motion on $M$) and function-theoretic properties of $M$ has been the
subject of an active area of research in the last decades. Deep
connections with the heat equation, Liouville type theorems,
capacity theory and spectral theory have been described, for
instance, in the beautiful survey \cite{Grigoryan}. In \cite{PRS4} and \cite{PRS5}, the authors showed that stochastic
completeness and parabolicity are also related to weak maximum
principles at infinity. This characterization reveals to be
fruitful in investigating many kinds of geometric problems (for a
detailed account, see \cite{PRS3}). Among the various conditions
equivalent to stochastic completeness, the following two are of prior interest to us:
\begin{itemize}
\item[-] [\emph{$L^\infty$-Liouville}] for some (any) $\lambda>0$, the sole bounded, non-negative, continuous weak solution of $\Delta u -\lambda u \ge 0$ is $u=0$;\vspace{0.1cm}
\item[-] [\emph{weak maximum principle}] for every $u\in C^2(M)$ with $u^\star=\sup_Mu<+\infty$, and for every $\eta<u^\star$,
\begin{equation}\label{wmpclassic}
\inf_{\Omega_\eta} \Delta u \le 0, \qquad \text{where } \ \Omega_\eta= u^{-1}\{(\eta,+\infty)\}.
\end{equation}
\end{itemize}
%
%
%
R.Z. \Ka \cite{khas} has found the following condition
for stochastic completeness. We recall that $w\in C^0(M)$ is
called an exhaustion if it has compact sublevels
$w^{-1}((-\infty,t])$, $t\in \R$. 
\begin{teo}[Khas'minskii test, \cite{khas}]
Suppose that there exists a compact set $K$ and a function $w\in C^0(M)\cap C^2(M\setminus K)$ satisfying for some $\lambda>0$:
$$
(i)\quad w \ \text{is an exhaustion;} \qquad (ii) \quad  \Delta w
-\lambda w \le 0 \ \text{ on } M\backslash K.
$$
Then $M$ is stochastically complete.
\end{teo}
A very similar characterization holds for the parabolicity of $M$.
Namely, among many others, parabolicity is equivalent to:
\begin{itemize}
\item[-] every bounded, non-negative continuous weak solutions of
$\Delta u\ge 0$ on $M$ is constant;\vspace{0.1cm}
\item[-] for every non-constant $u\in C^2(M)$ with $u^\star=\sup_Mu<+\infty$, and for every $\eta<u^\star$,
\begin{equation}\label{parclassic}
\inf_{\Omega_\eta} \Delta u < 0, \qquad \text{where } \ \Omega_\eta= u^{-1}\{(\eta,+\infty)\}.
\end{equation}
\end{itemize}
Note that the first condition is precisely case $\lambda=0$ of the
Liouville property above. As for \Ka type conditions, it has been proved by M. Nakai \cite{Nakai} and
Z. Kuramochi \cite{Kuramochi} that the parabolicity of $M$ is
indeed \emph{equivalent} to the existence of a so-called Evans
potential, that is, an exhaustion, harmonic function $w$ defined outside a
compact set $K$ and such that $w=0$ on $\partial K$. To the best
of our knowledge, an analogue of such equivalence for stochastic
completeness or for the nonlinear case has still to be proved, and this is the starting point of the present work.

With some modifications, it is possible to define the Liouville property, the \Ka test and Evans potentials also for $p$-Laplacians or other nonlinear operators, and the aim of this paper is to prove that in this more general setting the Liouville property is equivalent to the \Ka test, answering in the affirmative to a question raised in \cite{PRS1} (question 4.6). After that, a brief discussion on the connection with appropriate definitions of the weak maximum principle is included. The final section will be devoted to the existence of Evans type potentials in the particular setting of radially symmetric manifolds.
To fix the ideas, we state the main theorem in the ``easy case'' of the $p$-Laplacian, and then introduce the more general (and more technical) operators to which our theorem applies. Recall that for a function $u\in W^{1,p}_\loc(\Omega)$, the $p$-laplacian $\Delta_p$ is defined weakly as:
\begin{gather}\label{defplapl}
 \int_\Omega \phi \Delta_p u = -\int_\Omega \abs{\nabla u}^{p-2} \ps{\nabla u}{\nabla \phi}
\end{gather}
where $\phi\in \cc$ and integration is with respect to the Riemannian measure.
%
%
\begin{teo}\label{mainteolapla}
Let $M$ be a Riemannian manifold and let $p>1$, $\lambda \ge 0$.
Then, the following conditions are equivalent.
\begin{itemize}
 \item[$(W)$] The weak maximum principle for $C^0$
 holds for $\Delta_p$, that is, for every non-constant $u\in
 C^0(M)\cap W^{1,p}_{loc}(M)$ with $u^\star =\sup_M u <\infty$
 and for every $\eta<u^\star$ we have:
 \begin{gather}\label{wmpdef}
  \inf_{\Omega_\eta} \Delta_p u \leq 0 \qquad (<0 \text{ if }
 \lambda=0)
 \end{gather}
 weakly on $\Omega_\eta = u^{-1}\{(\eta, +\infty)\}$.
\item[$(L)$] Every non-negative, $L^\infty\cap W^{1,p}_\loc$
solution $u$ of $\Delta_p u -\lambda u^{p-1} \ge 0$ is constant
(hence zero if $\lambda>0$).
\item[$(K)$] For every compact $K$ with smooth boundary, there
exists an exhaustion $w \in C^0(M) \cap
W^{1,p}_\loc(M)$ such that
$$
w>0 \ \text{ on } M\backslash K, \quad w=0 \ \text{ on }
K, \quad \Delta_pw -\lambda w^{p-1} \le 0 \ \text{ on }
M\setminus K.
$$
\end{itemize}
\end{teo}
Up to some minor changes, the implications $(W)\Leftrightarrow (L)$ and $(K)\Rightarrow (L)$ have been shown in \cite{PRS2}, Theorem A, where it is also proved that, in $(W)$ and $(L)$, $u$ can be equivalently restricted to the class $C^1(M)$. In this respect, see also \cite{PRS1}, Section 2. On the other hand, the second author in \cite{valto} has proved that $(L) \Rightarrow (K)$ when $\lambda=0$. The proof developed in this article covers both the case $\lambda=0$ and $\lambda>0$, is easier and more straightforward and, above all, does not depend on some features which are typical of the $p$-Laplacian.

\section{Definitions and main theorems}\label{sec_deph}

\begin{notaz}
\emph{We set $\R^+=(0,+\infty)$, $\R^+_0=[0,+\infty)$, and $\R^-$, $\R^-_0$ accordingly; for a function $u$ defined on some set $\Omega$, $u^\star= \esssup_\Omega u$ and $u_\star=\essinf_\Omega u$; we will write $K\Subset \Omega$ whenever the set $K$ has compact closure in $\Omega$; $\lip_\loc(M)$ denotes the class of locally Lipschitz functions on $M$; with $u\in\hol_\loc(M)$ we mean that, for every $\Omega \Subset M$, $u \in C^{0,\alpha}(\Omega)$ for some $\alpha \in (0,1]$ possibly depending on $\Omega$. Finally, we will adopt the symbol $Q \doteq \ldots$ to define the quantity $Q$ as $\ldots$.}
\end{notaz}
In order for our techniques to work, we will consider quasilinear operators of the following
form. Let $A: TM \ra TM$ be a Caratheodory map, that is if
$\pi:TM\ra M$ is the bundle projection, $\pi \circ A = \pi$, moreover every representation $\tilde A$ of $A$ in local charts satisfies
\begin{itemize}
 \item $ \tilde A(x,\cdot)$ continuous for a.e. $x\in M$
 \item $\tilde A(\cdot,v)$ measurable for every $v\in \R^m$
\end{itemize}
Note that every continuous bundle map satisfies these assumptions. Furthermore, let $B : M\times \R \rightarrow \R$ be
of Caratheodory type, that is, $B(\cdot,t)$ is measurable for
every fixed $t\in \R$, and $B(x,\cdot)$ is continuous for a.e.
$x\in M$. We shall assume that there exists $p>1$ such that, for
each fixed open set $\Omega \Subset M$, the
following set of assumptions $\Se$ is met:

\begin{align}
&\ps{A(X)}{X} \ge a_1|X|^p \quad \forall \ X\in TM
\tag{A1}\label{A1}
\\[0.1cm]
&|A(X)| \le a_2 |X|^{p-1} \quad \forall X\in TM \tag{A2}
\label{A2}
\\[0.1cm]
& \begin{array}{l} A \text{ is strictly monotone, i.e. } \
\ps{A(X)-A(Y)}{X-Y}_p \ge 0 \ \text{ for} \\[0.1cm]
\text{every } x\in M, \ X,Y \in T_xM, \text{ with equality if and
only if } X=Y \end{array}
\tag{Mo} \label{M}\\[0.1cm]
&|B(x,t)| \le b_1 + b_2|t|^{p-1} \quad \text{for } t\in \R
\tag{B1} \label{B1}
\\[0.1cm]
&\text{for a.e. } x, \ B(x,\cdot) \text{ is monotone non-decreasing} \tag{B2} \label{B2}\\[0.1cm]
&\text{for a.e. } x, \ B(x,t)t \ge 0,
\tag{B3}\label{B3}
\end{align}
where $a_1,a_2,b_1,b_2$ are positive constants possibly depending
on $\Omega$. As explained in remark \ref{rem_B}, we could state our main theorem relaxing condition \ref{B1} to:
\begin{gather}
 |B(x,t)| \le b(t)  \quad \text{for } t\in \R \tag{B1+} \label{B1+}
\end{gather}
for some positive and finite function $b$, however for the moment we assume \ref{B1} to avoid some complications in the notation, and explain later how to extend our result to this more general case.

We define the operators $\oF,\oA,\oB : \wup \ra
\wupst$ by setting
\begin{equation}\label{definABF}
\begin{array}{llcl}
\oA : & u & \longmapsto & \Big[ \phi \in W^{1,p}(\Omega)
\longmapsto \int_\Omega \ps{A(\nabla u)}{\nabla \phi} \Big]
\\[0.5cm]
\oB : & u & \longmapsto & \Big[ \phi \in W^{1,p}(\Omega)
\longmapsto \int_\Omega B(x,u(x))\phi \Big]
\\[0.4cm]
\oF \doteq \oA+\oB. & & &
\end{array}
\end{equation}
With these assumptions, it can be easily verified that both $\oA$ and
$\oB$ map to continuous linear functionals on $W^{1,p}(\Omega)$
for each fixed $\Omega \Subset M$. We define the operators
$L_{\oA}$, $L_{\oF}$ according to the distributional equality:
$$
\int_M \phi L_{\oA}u \doteq - < \oA(u),\phi>, \quad \int_M \phi
L_{\oF}u \doteq - < \oF(u),\phi>
$$
for every $u\in W^{1,p}_\loc(M)$ and $\phi\in C^\infty_c(M)$,
where $<,>$ is the duality. In other words, in the weak sense
$$
L_{\oF} u = \diver(A(\nabla u)) - B(x,u) \qquad \forall \ u\in
W^{1,p}_\loc(M).
$$
\begin{ex}
\emph{The $p$-Laplacian defined in \eqref{defplapl}, corresponding to
the choices $A(X) \doteq |X|^{p-2}X$ and $B(x,t) \doteq 0$, satisfies all
the assumptions in $\Se$ for each $\Omega\Subset M$. Another
admissible choice of $B$ is $B(x,t) \doteq \lambda |t|^{p-2}t$, where
$\lambda \ge 0$. For such a choice,
\begin{equation}\label{plaplaconpote}
L_{\oF}u = \Delta_p u - \lambda|u|^{p-2}u
\end{equation}
is the operator of Theorem \ref{mainteolapla}. We stress that,
however, in $\Se$ we require no homogeneity condition either on
$A$ or on $B$.}
\end{ex}
\begin{ex}\label{ex2}
\emph{More generally, as in \cite{PRS2} and in \cite{PSZ}, for each function $\varphi\in C^0(\R^+_0)$ such that
$\varphi>0$ on $\R^+$, $\varphi(0)=0$, and for each symmetric,
positive definite $2$-covariant continuous tensor field $h \in
\Gamma(\mathrm{Sym}_2(TM))$, we can consider differential
operators of type
$$
L_{\varphi,h}u \doteq \diver \left( \frac{\varphi(|\nabla u|)}{|\nabla
u|} h(\nabla u, \cdot)^\sharp \right),
$$
where $\sharp$ is the musical isomorphism. Due to the continuity
and the strict positivity of $h$, the conditions \eqref{A1} and
\eqref{A2} in $\Se$ can be rephrased as
\begin{equation}\label{doublebound}
a_1t^{p-1} \le \varphi(t) \le a_2t^{p-1}.
\end{equation}
Furthermore, if $\varphi \in C^1(\R^+)$, a sufficient condition
for \eqref{M} to hold is given by
\begin{equation}\label{suffmono}
\frac{\varphi(t)}{t}h(X,X) + \left( \varphi'(t) -
\frac{\varphi(t)}{t}\right) \ps{Y}{X}h(Y,X) >0
\end{equation}
for every $X,Y$ with $|X|=|Y|=1$. The reason why it implies the
strict monotonicity can be briefly justified as follows: for
$L_{\varphi,h}$, \eqref{M} is equivalent to requiring
\begin{equation}\label{monophilapla}
\frac{\varphi(|X|)}{|X|}h(X,X-Y)-\frac{\varphi(|Y|)}{|Y|}h(Y,X-Y)
>0 \qquad \text{if } \ X\neq Y.
\end{equation}
In the nontrivial case when $X$ and $Y$ are not proportional, the
segment $Z(t)=Y+t(X-Y)$, $t\in [0,1]$ does not pass through zero, so that
$$
F(t) = \frac{\varphi(|Z|)}{|Z|}h(Z, Z')
$$
is $C^1$. Condition \eqref{suffmono} implies that $F'(t)>0$.
Hence, integrating we get $F(1)>F(0)$, that is,
\eqref{monophilapla}. We observe that, if $h$ is the metric
tensor, the strict monotonicity is satisfied whenever $\varphi$ is
strictly increasing on $\R^+$ even without any differentiability
assumption on $\varphi$.}
\end{ex}
\begin{ex}
\emph{Even more generally, if $A$ is of class $C^1$, a sufficient
condition for the monotonicity of $A$ has been considered in
\cite{antoninimugnaipucci}, Section 5 (see the proof of Theorem
5.3). Indeed, the authors required that, for every $x\in M$ and
every $X\in T_xM$, the differential of the map $A_x : T_xM\ra
T_xM$ at the point $X\in T_xM$ is positive definite as a linear
endomorphism of $T_X(T_xM)$. This is the analogue, for Riemannian
manifolds, of Proposition 2.4.3 in \cite{pucciserrin}.}
\end{ex}

We recall the concept of subsolutions and supersolutions for
$L_{\oF}$.
\begin{deph}
We say that $u \in W^{1,p}_\loc(M)$ solves $L_{\oF}u \ge 0$ (resp.
$\le 0$, $=0$) weakly on $M$ if, for every non-negative $\phi\in
C^\infty_c(M)$, $< \oF(u), \phi > \le 0$, (resp., $\ge 0$, $=0$). Explicitly, 
$$
\int_M \ps{A(\nabla u)}{\nabla \phi} + \int_M B(x,u)\phi \le 0 \ \ \ \text{(resp., } \ge 0, \, = 0\text{).}
$$
Solutions of $L_{\oF}u \ge 0$ (resp, $\le 0$, $=0$) are called
(weak) subsolutions (resp. supersolutions, solutions) for $L_\oF$.
\end{deph}
\begin{oss}
\emph{When defining solutions of $L_\oF u=0$, we can drop the requirement that the test function $\phi$ is non-negative. This can be easily seen by splitting $\phi$ into its positive and negative parts and using a density argument.}
\end{oss}

\begin{oss}
\emph{Note that, since $B$ is Caratheodory, \eqref{B3} implies that $B(x,0)=0$ a.e. on $M$. Therefore, the constant function $u=0$ solves $L_{\oF}u=0$. Again by $\eqref{B3}$, positive constants are supersolutions.}
\end{oss}

Following \cite{PRS2} and \cite{PRS1}, we present the analogues of the
$L^\infty$-Liouville property and the \Ka property for the
nonlinear operators constructed above.
\begin{deph}
Let $M$ be a Riemannian manifold, and let $\oA, \oB, \oF$ be as
above.
\begin{itemize}
\item[-] We say that the $L^\infty$-Liouville property $(L)$ for
$L^\infty$ (respectively, $\hol_\loc$) functions holds for the
operator $L_{\oF}$ if every $u\in L^\infty(M)\cap
W^{1,p}_{loc}(M)$ (respectively, $\hol_\loc(M) \cap
W^{1,p}_\loc(M)$) essentially bounded, satisfying $u\ge 0$ and
$L_{\oF}u \ge 0$ is constant.
\item[-] We say that the \Ka property $(K)$ holds for $L_{\oF}$
if, for every pair of open sets $K \Subset \Omega \Subset M$ with
Lipschitz boundary, and every $\eps>0$, there exists an exhaustion
function
$$
w \in C^0(M) \cap
W^{1,p}_\loc(M)
$$
such that
$$
\begin{array}{ll}
w>0 \ \text{ on } M\backslash K, & \quad w=0 \ \text{ on } K, \\[0.2cm]
w \le \eps \ \text{ on } \Omega\backslash K, & \quad L_{\oF}w \le
0 \ \text{on } M\backslash K.
\end{array}
$$
A function $w$ with such properties will be called a \Ka potential relative to the triple
$(K,\Omega,\eps)$.
\item[-] a \Ka potential $w$ relative to some triple
$(K,\Omega,\eps)$ is called an Evans potential if $L_{\oF}w=0$ on
$M\backslash K$. The operator $L_{\oF}$ has the Evans property
$(E)$ if there exists an Evans potential for every triple
$(K,\Omega,\eps)$.
\end{itemize}
\end{deph}
%
%
%
The main result in this paper is the following
\begin{teo}\label{mainteo}
Let $M$ be a Riemannian manifold, and let $A,B$ satisfy the set of
assumptions $\Se$, with \eqref{B1+} instead of \eqref{B1}. Define $\oA,\oB,\oF$ as in \eqref{definABF},
and $L_{\oA},L_{\oF}$ accordingly. Then, the conditions $(L)$ for
$\hol_\loc$, $(L)$ for $L^\infty$ and $(K)$ are equivalent.
\end{teo}
\begin{oss}
\emph{It should be observed that if $L_{\oF}$ is homogeneous, as
in \eqref{plaplaconpote}, the \Ka condition considerably
simplifies as in $(K)$ of Theorem \ref{mainteolapla}. Indeed, the
fact that $\delta w$ is still a supersolution for every
$\delta>0$, and the continuity of $w$, allow to get rid of
$\Omega$ and $\eps$.}
\end{oss}
Next, in Section \ref{sec_WMP} we briefly describe in which way $(L)$ and $(K)$ are related to the
concepts of weak maximum principle and parabolicity. Such relationship has been deeply investigated in \cite{PRS3}, \cite{PRS2}, whose ideas and proofs we will follow closely. With the aid of Theorem \ref{mainteo}, we will be able to prove the next Theorem \ref{LKW}. To state it, we shall restrict to a particular class of potentials $B(x,t)$, those of the form $B(x,t)=b(x)f(t)$ with
\begin{equation}\label{nuovopot}
\begin{array}{l}
\disp b, b^{-1} \in L^\infty_\loc(M), \quad b > 0 \text{ a.e. on }
M; \\[0.2cm]
f \in C^0(\R), \quad f(0)=0, \quad f \text{ is non-decreasing on }
\R.
\end{array}
\end{equation}
Clearly, $B$ satisfies \eqref{B1+}, \eqref{B2} and \eqref{B3}. As for $A$, we require \eqref{A1} and \eqref{A2}, as before.
\begin{deph}
Let $A,B$ be as above, define $\oA,\oB,\oF$ as in \eqref{definABF}
and $L_{\oA},L_{\oF}$ accordingly.
\begin{itemize}
\item[$(W)$] We say that $b^{-1}L_{\oA}$ satisfies the weak
maximum principle for $C^0$ functions if, for every $u\in
C^0(M)\cap W^{1,p}_\loc(M)$ such that $u^\star<+\infty$, and for
every $\eta<u^\star$,
$$
\inf_{\Omega_\eta} b^{-1}L_{\oA}u \le 0 \qquad \text{weakly on } \
\Omega_\eta = u^{-1}\{(\eta, +\infty)\}.
$$
\item[$(W_\pa)$] We say that $b^{-1}L_{\oA}$ is parabolic if, for
every non-constant $u\in C^0(M)\cap W^{1,p}_\loc(M)$ such that
$u^\star<+\infty$, and for every $\eta<u^\star$,
$$
\inf_{\Omega_\eta} b^{-1}L_{\oA}u < 0 \qquad \text{weakly on } \
\Omega_\eta = u^{-1}\{(\eta, +\infty)\}.
$$
\item[-] We say that $\oF$ is of type $1$ if, in the potential
$B(x,t)$, the factor $f(t)$ satisfies $f>0$ on $\R^+$. Otherwise,
when $f=0$ on some interval $[0,T]$, $\oF$ is called of type $2$.
\end{itemize}
\end{deph}
\begin{oss}
\emph{$\inf_{\Omega_\eta} b^{-1}L_{\oA}u \le 0$ weakly means that,
for every $\eps>0$, there exists $0\le \phi \in
C^\infty_c(\Omega_\eta)$, $\phi \not\equiv 0$ such that
$$
- <\oA(u), \phi> \  <  \eps \int b \phi.
$$
Similarly, with $\inf_{\Omega_\eta} b^{-1}L_{\oA}u < 0$ weakly we
mean that there exist $\eps>0$ and $0\le \phi\in
C^\infty_c(\Omega_\eta)$, $\phi \not\equiv 0$ such that $-
<\oA(u), \phi> \ < -\eps \int b \phi$. }
\end{oss}
\begin{teo}\label{LKW}
Under the assumptions \eqref{nuovopot} for $B(x,t)=b(x)f(t)$, and \eqref{A1}, \eqref{A2} for $A$, the following properties are equivalent:
\begin{itemize}
\item[-] The operator $b^{-1}L_{\oA}$ satisfies $(W)$;
\item[-] Property $(L)$ holds for some (hence any) operator $\oF$
of type 1;
\item[-] Property $(K)$ holds for some (hence any) operator $\oF$
of type 1;
\end{itemize}
Furthermore, under the same assumptions, the next equivalence holds:
\begin{itemize}
\item[-] The operator $b^{-1}L_{\oA}$ is parabolic;
\item[-] Property $(L)$ holds for some (hence any) operator $\oF$
of type 2;
\item[-] Property $(K)$ holds for some (hence any) operator $\oF$
of type 2;
\end{itemize}

\end{teo}

In the final Section \ref{sec_Evans}, we address the question whether $(W)$, $(K)$, $(L)$ are equivalent to the Evans property $(E)$. Indeed, it should be observed that, in Theorem \ref{mainteo}, no growth control on $B$ as a function of $t$ is required at all. On the contrary, as we will see, the validity of the Evans property forces some precise upper bound for its growth. To better grasp what we shall expect, we will restrict to the case of radially symmetric manifolds. For the statements of the main results, we refer the reader directly to Section \ref{sec_Evans} covering the situation.

\section{Technical tools}\label{sec_tech}

In this section we introduce some technical tools, such as the
obstacle problem, that will be crucial to the proof of our main theorems. In doing so, a number of basic results from literature is recalled. We have decided to add a full proof to those results for which we have not found any reference convering the situation at hand. Our aim is to keep the paper basically self-contained, and to give the non-expert reader interested in this topic a brief overview also of the standard technical tricks. Throughout this section, we will always assume that the
assumptions in $\Se$ are satisfied, if not explicitly stated.
First, we state some basic results on subsolutions-supersolutions such as the comparison principle, which follows from the
monotonicity of $A$ and $B$.
\begin{prop}\label{teo_comp}
Assume $w$ and $s$ are a super and a subsolution defined on
$\Omega$. If $\min\{w-s,0\}\in \wupz$, then $w\geq s$ a.e. in
$\Omega$.
\end{prop}
\begin{proof}
 This theorem and its proof, which follows quite easily using the right test function in the definition of supersolution, are standard in potential theory. For a detailed proof see \cite{antoninimugnaipucci}, Theorem 4.1.
\end{proof}
Next, we observe that $A$, $B$ satisfy all the assumptions for the
subsolution-supersolution method in \cite{kuratake} to be
applicable.

\begin{teo}[\cite{kuratake}, Theorems 4.1, 4.4 and 4.7]\label{subsuper}
Let $\phi_1,\phi_2 \in L^\infty_\loc \cap W^{1,p}_\loc$ be, respectively, a subsolution and a supersolution for $L_{\oF}$ on $M$, and suppose that $\phi_1\le \phi_2$ a.e. on $M$. Then, there is a solution $u\in L^\infty_\loc\cap W^{1,p}_\loc$ of $L_{\oF}u=0$ satisfying $\phi_1\le u\le \phi_2$ a.e. on $M$. 
\end{teo}

A fundamental property is the strong maximum principle, which
follows from the next Harnack inequality
\begin{teo}[\cite{pucciserrin}, Theorems 7.1.2, 7.2.1 and
7.4.1]\label{harnack}
Let $u\in W^{1,p}_\loc(M)$ be a non-negative solution of $L_{\oA}u
\le 0$. Let the assumptions in $\Se$ be satisfied. Fix a
relatively compact open set $\Omega \Subset M$.
\begin{itemize}
\item[(i)] Suppose that $1<p\le m$, where $m=\mathrm{dim}M$. Then,
for every ball $B_{4R} \subset \Omega$ and for every $s \in (0,
(p-1)m/(m-p))$, there exists a constant $C$ depending on $R$, on the
geometry of $B_{4R}$, on $m$ and on the parameters $a_1,a_2$ in
$\Se$ such that
$$
\norm{u}_{L^s(B_{2R})} \le C \Big(\essinf_{B_{2R}}u\Big).
$$
\item[(i)] Suppose that $p>m$. Then, for every ball $B_{4R}
\subset \Omega$, there exists a constant $C$ depending on $R$, on
the geometry of $B_{4R}$, on $m$ and on the parameters $a_1,a_2$
in $\Se$ such that
$$
\esssup_{B_R} u \le C \Big(\essinf_{B_R}u\Big).
$$
\end{itemize}
In particular, for every $p>1$, each non-negative solution $u$ of
$L_{\oA}u \le 0$ on $M$ is such that either $u=0$ on $M$ or
$\essinf_\Omega u >0$ for every relatively compact set $\Omega$.
\end{teo}
\begin{oss}
\emph{We spend few words to comment on the Harnack inequalities
quoted from \cite{pucciserrin}. In our assumptions $\Se$, the
functions $\bar{a}_2, \bar{a}, b_1,b_2,b$ in Chapter 7, (7.1.1)
and (7.1.2) and the function $a$ in the monotonicity inequality
(6.1.2) can be chosen to be identically zero. Thus, in Theorems
7.1.2 and 7.4.1 the quantity $k(R)$ is zero. This gives no
non-homogeneous term in the Harnack inequality, which is essential
for us. For this reason, we cannot weaken \eqref{A2} to
$$
|A(X)|\le a_2|X|^{p-1} + \bar{a}
$$
locally on $\Omega$, since the presence of non-zero $\bar{a}$
implies that $k(R)>0$. It should be observed that Theorem 7.1.2 is
only stated for $1<p<m$ but, as observed at the beginning of
Section 7.4, the proof can be adapted to cover the case $p=m$.}
\end{oss}
\begin{oss}\label{ossfilapla}
\emph{In the rest of the paper, we will only use the fact that
either $u\equiv 0$ or $u>0$ on $M$, that is, the strong maximum
principle. It is worth observing that, for the operators $L_{\oA}=L_{\varphi,h}$
described in Example \ref{ex2}, very general strong maximum
principles for $C^1$ or $\lip_\loc$ solutions of $L_{\varphi,h}u
\le 0$ on Riemannian manifolds have been obtained in
\cite{puccirigoliserrin} (see Theorem 1.2 when $h$ is the metric
tensor, and Theorems 5.4 and 5.6 for the general case). In
particular, if $h$ is the metric tensor, the sole requirements
\begin{equation}\label{ipophi}
\varphi \in C^0(\R^+_0), \ \ \varphi(0)=0, \ \ \varphi >0 \text{
on } \R^+, \ \  \varphi \text{ in strictly increasing on } \R^+
\end{equation}
are enough for the strong maximum principle to hold for $C^1$
solutions of $L_\varphi u \le 0$. Hence, for instance for
$L_\varphi$, the two-sided bound \eqref{doublebound} on $\varphi$
can be weakened to any bound ensuring that the comparison and
strong maximum principles hold, the subsoluton-supersolution
method is applicable and the obstacle problem has a solution. For
instance, besides \eqref{ipophi}, the requirement
\begin{equation}\label{buonefi}
\varphi(0)=0, \quad a_1 t^{p-1} \le \varphi(t) \le a_2t^{p-1} +
a_3
\end{equation}
is enough for Theorems, \ref{teo_comp}, \ref{subsuper}, and it
also suffices for the obstacle problem to admit a unique solution,
as the reader can infer from the proof of the next Theorem
\ref{existostacolo}.}
\end{oss}
\begin{oss}\label{ossregularity}
\emph{Regarding the above observation, if $\varphi$ is merely
continuous then even solutions of $L_\varphi u=0$ are not expected
to be $C^1$, nor even $\lip_\loc$. Indeed, in our assumptions the
optimal regularity for $u$ is (locally) some H\"older class, see
the next Theorem \ref{regularity}. If $\varphi \in C^1(\R^+)$ is
more regular, then we can avail of the regularity result in
\cite{tolksdorf} to go even beyond the $C^1$ class. Indeed, under
the assumptions
$$
\gamma(k+ t)^{p-2} \le
\min\left(\varphi'(t),\frac{\varphi(t)}{t}\right) \le
\max\left(\varphi'(t),\frac{\varphi(t)}{t}\right) \le
\Gamma(k+t)^{p-2},
$$
for some $k\ge 0$ and some positive constants $\gamma\le \Gamma$,
then each solution of $L_\varphi u=0$ is in some class
$C^{1,\alpha}$ on each relatively compact set $\Omega$, where
$\alpha \in (0,1)$ may depend on $\Omega$. When $h$ is not the
metric tensor, the condition on $\varphi$ and $h$ is more
complicated, and we refer the reader to \cite{PRS2} (in
particular, see (0.1) $(v)$, $(vi)$ p. 803).}
\end{oss}
Part of the regularity properties that we need are summarized in the following
\begin{teo}\label{regularity}
Let the assumptions in $\Se$ be satisfied.
\begin{itemize}
\item[(i)] [\cite{ZM}, Theorem 4.8] If $u$ solves $L_\oF u \le 0$
on some open set $\Omega$, then there exists a representative in
$\wup$ which is lower semicontinuous.
\item[(ii)] [\cite{ladyura}, Theorem 1.1 p. 251] If $u\in
L^\infty(\Omega) \cap \wup$ is a bounded solution of $L_{\oF}u=0$
on $\Omega$, then there exists $\alpha \in (0,1)$ depending on the
geometry of $\Omega$, on the constants in $\Se$ and on
$\|u\|_{L^\infty(\Omega)}$ such that $u\in C^{0,\alpha}(\Omega)$.
Furthermore, for every $\Omega_0\Subset \Omega$, there exists
$C=C(\gamma, \mathrm{dist}(\Omega_0,
\partial \Omega))$ such that
$$
\|u\|_{C^{0,\alpha}(\Omega_0)} \le C.
$$
\end{itemize}
\end{teo}
\begin{oss}
\emph{As for $(i)$, it is worth observing that, in our
assumptions, both $b_0$ and $\textbf a$ in the statement of
\cite{ZM}, Theorem 4.8 are identically zero. Although we will not need the following properties, it is worth noting that
any $u$ solving $L_{\oF}u \le 0$ has a Lebesgue point everywhere and
is also $p$-finely continuous (where finite).}
\end{oss}
Next, this simple elliptic estimate for
locally bounded supersolutions is useful:
\begin{prop}\label{caccio}
Let $u$ be a bounded solution of $L_{\oF}u \le 0$ on $\Omega$. Then, for every
relatively compact, open set $\Omega_0\Subset \Omega$ there is a
constant $C>0$ depending on $p, \, \Omega, \, \Omega_0$ and on the
parameters in $\Se$ such that
$$
\norm{\nabla u}_{L^p(\Omega_0)} \le C(1+\norm{u}_{L^\infty(\Omega)})
$$
\end{prop}
\begin{proof}
Given a supersolution $u$, the monotonicity of $B$ assures that for every positive constant $c$ also $u+c$ is a supersolution, so without loss of generality we may assume that $u_\star \ge 0$. Thus, $u^\star=\norm{u}_{L^\infty(\Omega)}$.\\
Shortly, with $\|\cdot \|_p$ we denote the $L^p$ norm on $\Omega$,
and with $C$ we denote a positive constant depending on $p,\Omega$
and on the parameters in $\Se$, that may vary from place to place.
Let $\eta\in C^\infty_c(\Omega)$ be such that $0\le \eta \le 1$ on
$\Omega$ and $\eta=1$ on $\Omega_0$. Then, we use the non-negative
function $\phi= \eta^p(u^\star-u)$ in the definition of
supersolution to get, after some manipulation and from \eqref{A1},
\eqref{A2} and \eqref{B3},
\begin{equation}\label{prima}
\begin{array}{lcl}
\disp a_1\int_\Omega \eta^p|\nabla u|^p & \le & \disp pa_2
\int_\Omega |\nabla u|^{p-1}\eta^{p-1}(u^\star-u)|\nabla \eta| + \int_\Omega \eta^p B(x,u)u^\star \\[0.4cm]
\end{array}
\end{equation}
Using \eqref{B1}, the integral involving $B$ is roughly estimated
as follows:
\begin{equation}\label{seconda}
\int_\Omega \eta^p B(x,u)u^\star \le |\Omega|(b_1u^\star +
b_2(u^\star)^p) \le C(1+ u^\star)^p,
\end{equation}
where the last inequality follows by applying Young inequality on
the first addendum. As for the term involving $|\nabla \eta|$,
using $(u^\star-u)\le u^\star$ and again Young inequality $|ab|
\le |a|^p/(p\eps^p)+ \eps^q|b|^q/q$ we obtain
\begin{equation}\label{terza}
\begin{array}{lcl}
\disp pa_2 \int_\Omega \big(|\nabla
u|^{p-1}\eta^{p-1}(u^\star-u)|\nabla \eta|\big) & \le & \disp pa_2
\int_\Omega \big(|\nabla
u|^{p-1}\eta^{p-1}\big)\big(u^\star|\nabla \eta|\big) \\[0.4cm]
& \le & \frac{a_2}{\eps^p} \norm{\eta \nabla u}_p^p +
\frac{a_2p\eps^q}{q} \norm{\nabla \eta}^p_p(u^\star)^p
\end{array}
\end{equation}
Choosing $\eps$ such that $a_2\eps^{-p}=a_1/2$, inserting
\eqref{seconda} and \eqref{terza} into \eqref{prima} and
rearranging we obtain
$$
\frac{a_1}{2} \norm{\eta\nabla u}_p^p \le C \Big[1+
(1+\norm{\nabla \eta}_p^p)(u^\star)^p \Big].
$$
Since $\eta=1$ on $\Omega_0$ and $\norm{\nabla \eta}_p \le C$,
taking the $p$-root the desired estimate follows.
\end{proof}
\begin{oss}
\emph{We observe that, when $B\neq 0$ we cannot apply the
technique of \cite{HKM}, Lemma 3.27 to get a Caccioppoli-type
inequality for bounded, non-negative supersolutions. The reason is
that subtracting a positive constant to a supersolution does not
yield, for general $B\neq 0$, a supersolution. It should be
stressed that, however, when $p\le m$ a refined Caccioppoli
inequality for supersolution has been given in in \cite{ZM},
Theorem 4.4.}
\end{oss}
Now, we fix our attention on the obstacle problem. There are a lot
of references regarding this subject (for example see
\cite{ZM}, Chapter 5 or \cite{HKM}, Chapter 3 in the case $B =0$).
As often happens, notation can be quite different from one
reference to another. Here we try to adapt the conventions used in
\cite{HKM}, and for the reader's convenience we also sketch some
of the proofs.\\
First of all, some definitions. Given a function $\psi:\Omega\to
\R\cup{\pm\infty}$, and given $\theta\in W^{1,p}(\Omega)$, we
define the closed convex set
\begin{gather*}
 \K_{\psi,\theta} \doteq \{f\in \wup \ \vert \  \ f\geq \psi  \ \text{ a.e. and } \ f-\theta\in
\wupz\}.
\end{gather*}
Loosely speaking, $\theta$ determines the boundary condition for the
solution $u$, while $\psi$ is the ``obstacle''-function. Most of
the times, obstacle and boundary function coincide, and in this
case we use the convention $\K_{\theta}\doteq \K_{\theta,\theta}$.
We say that $u\in \Kpt$ solves the obstacle problem if for every
$\varphi\in \Kpt$:
\begin{gather}\label{eq_obs}
<\oF(u),\varphi-u> \ \geq 0.
\end{gather}
Note that for every nonnegative $\phi\in \cc$ the function $\varphi= u+\phi$
belongs to $\Kpt$, and this implies that the solution to the
obstacle problem is always a supersolution. Note also that if we
choose $\psi=-\infty$, we get the standard Dirichlet problem with
Sobolev boundary value $\theta$ for the operator $\F$, in fact in
this case any test function $\phi\in \cc$ verifies $u\pm
\phi\in \Kpt$, and so inequality in \eqref{eq_obs} becomes an
equality.
Next, we address the solvability of the obstacle problem.
\begin{teo}\label{existostacolo}
Under the assumptions $\Se$, if $\Omega$ is relatively compact and
$\K_{\psi,\theta}$ is nonempty, then there exists a unique
solution to the relative obstacle problem.
\end{teo}
\begin{proof}
The proof is basically the same if we assume $B=0$, as in
\cite{HKM}, Appendix 1; in particular, it is an application of
Stampacchia theorem, see for example Corollary
III.1.8 in \cite{KS}. To apply the theorem, we shall verify that $\Kpt$ is closed and convex, which follows straightforwardly from its very definition, and that
$\oF:\wup\to \wupst$ is weakly continuous, monotone and coercive.
Monotonicity is immediate by properties \eqref{M}, \eqref{B2}. To
prove that $\oF$ is weakly continuous, we take a sequence $u_i\to
u$ in $\wup$. By using \eqref{A2} and \eqref{B1}, we deduce from
\eqref{definABF} that
$$
|<\oF(u_i), \phi>| \le \Big((a_2+b_2)\norm{u_i}^{p-1}_{\wup} +
b_1|\Omega|^{\frac{p-1}{p}}\Big)\norm{\phi}_{\wup}
$$
Hence the $\wupst$ norm of $\{\F(u_i)\}$ is bounded. Since $\wupst$ is reflexive, we can extract from any
subsequence a weakly convergent sub-subsequence
$\oF(u_k) \rightharpoonup z$ in $\wupst$, for some $z$. From $u_k\ra
u$ in $\wup$, by Riesz theorem we get (up to a further subsequence) $(u_k, \nabla u_k) \rightarrow
(u,\nabla u)$ pointwise on $\Omega$, and since the maps
$$
X \longmapsto A(X), \qquad t \longmapsto B(x,t)
$$
are continuous, then necessarily $z=\oF(u)$. Since this is true
for every weakly convergent subsequence $\{\oF(u_k)\}$, we deduce
that the whole $\oF(u_i)$ converges weakly to $\oF(u)$. This
proves the weak continuity of $\oF$.\\
Coercivity on $\Kpt$ follows if we fix any $\varphi\in \Kpt$ and consider a diverging sequence $\{u_i\}\subset \Kpt$ and calculate:
\begin{gather*}
 \frac{\ps{\F(u_i)-\F(\varphi)}{u_i-\varphi}}{\norm{u_i-\varphi}_{\wup}}\stackrel{\eqref{B3}}{\geq} \frac{\ps{\A(u_i)-\A(\varphi)}{u_i-\varphi}}{\norm{u_i-\varphi}_{\wup}}\stackrel{\eqref{A1}, \eqref{A2}}{\geq}\\
\geq \frac{a_1\ton{\norm{\nabla u_i}_p^p + \norm{\nabla \varphi}_p^p}-a_2\ton{\norm{\nabla u_i}_p^{p-1}\norm{\nabla \varphi}_p+\norm{\nabla u_i}_p\norm{\nabla \varphi}_p^{p-1}}}{\norm{u_i-\varphi}_{\wup}}
\end{gather*}
This last quantity tends to infinity as $i$ goes to infinity thanks to the Poincarè inequality on $\Omega$:
\begin{gather*}
 \norm{u_i-\varphi}_{L^p(\Omega)}\leq C \norm{\nabla u_i - \nabla \varphi}_{L^p(\Omega)}
\end{gather*}
which leads to $\norm{\nabla u_i}_{L^p(\Omega)} \ge
C_1 + C_2\norm{u_i}_{W^{1,p}(\Omega)}$ for some constants $C_1, C_2$, where $C_1$
depends on $\norm{\varphi}_{\wup}$.
\end{proof}
A very important characterization of the solution of the obstacle
problem is a corollary to the following comparison, whose proof
follows closely that of the comparison Proposition \ref{teo_comp}.
\begin{prop}\label{compaosta}
If $u$ is a solution to the obstacle problem $\Kpt$, and if $w$ is
a supersolution such that $\min\{u,w\}\in \Kpt$, then $u\leq w $
a.e.
\end{prop}
\begin{proof}
Define $U=\{x \vert \ u(x)>w(x)\}$. Suppose by contradiction that
$U$ has positive measure. Since $u$ solves the obstacle problem,
using \eqref{eq_obs} with the function $\varphi=\min\{u,w\}\in \Kpt$
we get
\begin{equation}
0 \le \ < \oF(u), \varphi-u> \ = \int_U \ps{A(\nabla u)}{\nabla w
-\nabla u} + \int_U B(x,u)(w-u).
\end{equation}
On the other hand, applying the definition of supersolution $w$
with the test function $0\le \phi=u-\min\{u,w\} \in \wupz$ we
get
\begin{equation}
0 \le \ < \oF(w),\phi> \ = \int_U \ps{A(\nabla w)}{\nabla u
-\nabla w} + \int_U B(x,w)(u-w)
\end{equation}
adding the two inequalities we get, by \eqref{M} and
\eqref{B2},
$$
0 \le \int_U \ps{A(\nabla u)- A(\nabla w)}{\nabla w -\nabla u} +
\int_U \big[B(x,u)-B(x,w)\big](w-u) \le 0.
$$
%
Since $A$ is strictly monotone, $\nabla u = \nabla w$ a.e. on $U$,
so that $\nabla((u-w)_+)=0$ a.e. on $\Omega$. Consequently, since
$U$ has positive measure, $u-w=c$ a.e. on $\Omega$, where $c$ is a
positive constant. Since $\min\{u,w\}\in \Kpt$, we get $c=u-w=
u-\min\{u,w\} \in \wupz$, contradiction.
\end{proof}
\begin{cor}\label{corsuper}
The solution $u$ to the obstacle problem in $\Kpt$ is the smallest
supersolution in $\Kpt$.
\end{cor}

\begin{prop}\label{minsupmaxsub}
Let $w_1,w_2\in W^{1,p}_\loc(M)$ be supersolutions for $L_\oF$. Then, $w \doteq \min\{w_1,w_2\}$
is a supersolution. Analogously, if $u_1,u_2\in W^{1,p}_\loc(M)$ are
subsolutions for $L_\oF$, then so is $u \doteq \max\{u_1,u_2\}$.
\end{prop}
\begin{proof}
Consider a smooth exhaustion $\{\Omega_j\}$ of $M$, and the obstacle problem $\K_w$ on $\Omega_j$. By Corollary \ref{corsuper}
its solution is necessarily $w_{|\Omega_j}$, and so $w$ is a supersolution
being locally the solution of an obstacle problem. As for the second part
of the statement, define $\widetilde{A}(X) \doteq -A(-X)$ and
$\widetilde{B}(x,t) \doteq -B(x,-t)$. Then, $\widetilde{A},\widetilde{B}$
satisfy the set of assumptions $\Se$. Denote with
$\widetilde{\mathcal{F}}$ the operator associated to
$\widetilde{A},\widetilde{B}$. Then, it is easy to see that
$L_{\oF} u_i \ge 0$ if and only if
$L_{\widetilde{\mathcal{F}}}(-u_i)\le 0$, and to conclude it is
enough to apply the first part with operator
$L_{\widetilde{\mathcal{F}}}$.
\end{proof}

The next version of the pasting lemma generalizes the previous proposition to the case when one of the supersolutions is not defined on the whole $M$. Before stating it, we need a preliminary definition. Given an open subset $\Omega \subset M$, possibly with non-compact closure, we recall that the space $W^{1,p}_\loc(\overline{\Omega})$ is the set of all functions $u$ on $\Omega$ such that, for every relatively compact open set $V \Subset M$ that intersects $\Omega$, $u \in W^{1,p}(\Omega \cap V)$. A function $u$ in this space is, loosely speaking, well-behaved on relatively compact portions of $\partial \Omega$, while no global control on the $W^{1,p}$ norm of $u$ is assumed. Clearly, if $\Omega$ is relatively compact, $W^{1,p}_\loc(\overline{\Omega})= \wup$. We identify the following subset of $W^{1,p}_\loc(\overline{\Omega})$, which we call $X^p_0(\Omega)$:
\begin{equation}\label{definXp0}
X^p_0(\Omega) = \left\{ \begin{array}{l}
\disp u \in W^{1,p}_\loc(\overline{\Omega}) \text{ such that, for every open set } U \Subset M \text{ that} \\[0.3cm]
\text{intersects } \Omega \text{, there exists } \{\phi_n\}_{n=1}^{+\infty} \subset C^0(\overline{\Omega \cap U})\cap W^{1,p}(\Omega \cap U)\text{,} \\[0.3cm]
\disp \text{with } \phi_n \equiv 0 \text{ in a neighbourhood of } \partial \Omega \text{, satisfying} \\[0.3cm]
\disp \varphi_n \ra u \text{ in } W^{1,p}(\Omega \cap U) \text{ as } n\ra +\infty
\end{array} \right.
\end{equation}
If $\Omega$ is relatively compact, then $X_0^p(\Omega) = W^{1,p}_0(\Omega)$. 
\begin{oss}\label{bordopercontinue}
\emph{Observe that, if $u \in C^0(\overline{\Omega}) \cap W^{1,p}_\loc(\overline{\Omega})$, then $u \in X^p_0(\Omega)$ if and only if $u=0$ on $\partial \Omega$. This is the version, for non-compact domains $\Omega$, of a standard result. However, for the convenience of the reader we briefly sketch the proof. Up to working with positive and negative part separately, we can suppose that $u \ge 0$ on $\Omega$. If $u=0$ on $\partial \Omega$, then choosing the sequence $\phi_n = \max\{u- 1/n, 0\}$ it is easy to check that $u \in X^p_0(\Omega)$. Viceversa, if $u\in X^p_0(\Omega)$, let $x_0 \in \partial \Omega$ be any point. Choose $U_1 \Subset U_2\Subset M$ such that $x_0 \in U_1$, and a sequence $\{\phi_n\} \in C^0(\overline{\Omega \cap U_2}) \cap W^{1,p}(\Omega \cap U_2)$ as in the definition of $X^p_0(\Omega)$. If $\psi \in C^\infty_c(U_2)$ is a smooth cut-off function such that $\psi=1$ on $U_1$, then $\psi\phi_n \ra \psi u$ in $W^{1,p}(\Omega \cap U_2)$. Since $\psi \phi_n$ is compactly supported in $\Omega \cap U_2$, then $\psi u \in W^{1,p}_0(\Omega \cap U_2)$. It is a standard fact that, in this case, $\psi u=0$ on $\partial (\Omega \cap U_2)$. Since $x_0 \in \partial \Omega \cap U_2 \subset \partial(\Omega \cap U_2)$, $u(x_0) = u\psi(x_0)=0$. By the arbitrariness of $x_0$, this shows that $u=0$ on $\partial \Omega$.} 
\end{oss}
\begin{lemma}\label{pasting}
Let $w_1 \in W^{1,p}_\loc(M)$ be a supersolution for $L_\oF$, and let $w_2 \in W^{1,p}_\loc(\overline{\Omega})$ be a supersolution on some open set $\Omega$ with $\overline{\Omega}\subset M$, $\overline{\Omega}$ being possibly non-compact. Suppose that $\min\{w_2-w_1,0\} \in X^p_0(\Omega)$. Then, the function
$$
m\doteq \left\{ \begin{array}{ll} \min\{w_1,w_2\} & \quad \text{on } \Omega \\[0.1cm]
w_1 & \quad \text{on } M \backslash \Omega \end{array} \right.
$$ 
is a supersolution for $L_\oF$ on $M$. In particular, if further $w_1 \in C^0(M)$ and $w_2 \in C^0(\overline{\Omega})$, then $m$ is a supersolution on $M$ whenever $w_1=w_2$ on $\partial \Omega$. A similar statement is valid for subsolutions, replacing $\min$ with $\max$.
\end{lemma}
\begin{proof}
We first need to check that $m \in W^{1,p}_\loc(M)$. Let $U\Subset M$ be an open set. By assumption, there exists a sequence of functions $\{\phi_n\} \in C^0(\overline{\Omega \cap U}) \cap W^{1,p}(\Omega \cap U)$, each $\phi_n$ being zero in some neighbourhood of $\partial \Omega$, which converges in the $W^{1,p}$ norm to $\min\{w_2-w_1,0\}$. We can thus continuously extend $\phi_n$ on the whole $U$ by setting $\phi_n=0$ on $U\backslash \Omega$, and the resulting extension is in $W^{1,p}(U)$. Define $u= \min\{w_2-w_1,0\}\chi_{\Omega}$, where $\chi_{\Omega}$ is the indicatrix function of $\Omega$. Then, $\phi_n \rightarrow u$ in $W^{1,p}(U)$, so that $u\in W^{1,p}(U)$. It follows that $w_1 + \phi_n \in W^{1,p}(U)$ converges to $m=w_1+u$, which shows that $m\in W^{1,p}(U)$. To prove that $L_\oF m \le 0$ we use a technique similar to Proposition \ref{compaosta}. Let $U \Subset M$ be a fixed relatively compact open set, and let $s$ be the solution to the obstacle problem $\K_m$ on $U$. Then we have by Corollary \ref{corsuper} $s \leq w_1$ a.e. on  $U$ and so $s=w_1=m$ on $U\backslash\Omega$. Since $s$ solves the obstacle problem, using $\varphi=m$ in equation \eqref{eq_obs} we have:
\begin{equation}\label{picci}
0 \leq \ < \oF(s), m-s> \ = \int_{\Omega\cap U} \ps{A(\nabla s)}{\nabla m -\nabla s} + \int_{\Omega \cap U} B(x,s)(m-s).
\end{equation}
On the other hand $m$ is a supersolution in $\Omega \cap U$, being the minimum of two supersolutions, by Proposition \ref{minsupmaxsub}. To apply the weak definition of $L_\oF m \le 0$ on $\Omega \cap U$ to the test function $s-m$, 
%
we first claim that $s-m \in W^{1,p}_0(\Omega\cap U)$. Since we know that $s\leq w_1$ on $U$, then on $\Omega \cap U$
$$
0 \le s-m \leq w_1 -\min\{w_2,w_1\} = - \min\{w_2-w_1, 0\} \in X^p_0(\Omega). 
$$
The claim now follows by a standard result (see for example \cite{HKM}, Lemma 1.25), but for the sake of completeness we sketch the proof. Since $0 \le s-m\in W^{1,p}_0 (U)$ by the definition of the obstacle problem, there exists a sequence of nonnegative functions $\psi_n\in C^{\infty}_c(U)$ converging to $s-m$. We further consider the sequence $\{\phi_n\}$ of continuous functions, converging to $\min\{w_2-w_1,0\}$, defined at the beginning of this proof. Then, on $\Omega \cap U$, $0 \le s-m \le \lim_{n} \min\{-\phi_n,\psi_n\}$, where the limit is taken in $W^{1,p}(\Omega\cap U)$. Now, $\min\{-\phi_n,\psi_n\}$ has compact support in $\Omega \cap U$, and this proves the claim.
Applying the definition of $L_{\oF}m\le 0$ to the test function $s-m$ we get: 
\begin{equation}\label{pocci}
0\leq \ <\oF(m),s-m> \ = \int_{\Omega\cap U} \ps{A(\nabla m)}{\nabla s -\nabla m} + \int_{\Omega\cap U} B(x,m)(s-m).
\end{equation}
Summing inequalities \eqref{picci} and \eqref{pocci}, we conclude as in Proposition \ref{compaosta} that $\nabla(s-m)=0$ in $\Omega \cap U$ with $s-m\in W^{1,p}_0(\Omega\cap U)$, and so the two functions are equal there. Since $s=w=m$ on $U\backslash \Omega$, then $m=s$ is a supersolution on $U$. The thesis follows by the arbitrariness of $U$. If further $w_1 \in C^0(M)$ and $w_2 \in C^0(\overline{\Omega})$, then the conclusion follows by Remark \ref{bordopercontinue}. The proof of the statement for subsolutions is obtained via the same trick as in Proposition \ref{minsupmaxsub}.
\end{proof}

As for the regularity of solutions of the obstacle problem, we have
\begin{teo}[\cite{ZM}, Theorem 5.4 and Corollary 5.6]\label{continuosta}
If the obstacle $\psi$ is continuous in $\Omega$, then the
solution $u$ to $\Kpt$ has a continuous representative in the
Sobolev sense. Furthermore, if $\psi \in C^{0,\alpha}(\Omega)$ for
some $\alpha\in (0,1)$, then there exist $C,\beta>0$ depending
only on $p,\alpha,\Omega, \norm{u}_{L^\infty(\Omega)}$ and on the
parameters in $\Se$ such that
$$
\norm{u}_{C^{0,\beta}(\Omega)}\le
C(1+\norm{\psi}_{C^{0,\alpha}(\Omega)})
$$
\end{teo}
\begin{oss}
\emph{The interested reader should be advised that, in the
notation of \cite{ZM}, $b_0$ and $\textbf a$ are both zero with
our assumptions. Stronger results, for instance $C^{1,\alpha}$
regularity, can be obtained from stronger requirements on $\psi$,
$A$ and $B$ which are stated for instance in \cite{ZM}, Theorem
5.14.}
\end{oss}
In the proof of our main theorem, and to get some boundary regularity results, it will be important to see what happens on the set where the solution of the obstacle problem is strictly above the obstacle.
\begin{prop}\label{esoluzione}
 Let $u$ be the solution of the obstacle problem $\Kpt$ with continuous obstacle $\psi$. If $u>\psi$ on an open set
$D$, then $u$ is a solution of $L_\oF u=0$ on $D$.
\end{prop}
\begin{proof}
%
%
Consider any test function $\phi\in C^{\infty}_c(D)$. Since
$u>\psi$ on $D$, and since $\phi$ is bounded, by continuity
there exists $\delta>0$ such that $u\pm \delta\phi \in \Kpt$.
From the definition of solution to the obstacle problem we have
that:
\begin{gather*}
\pm  < \oF(u),\phi> \ = \frac{1}{\delta} < \oF(u), \pm \delta
\phi> \ =  \frac{1}{\delta} < F(u), (u\pm\delta\phi) -u > \
\ge 0,
\end{gather*}
hence $<\oF(u),\phi> \ =0$ for every $\phi\in C^\infty_c(D)$,
as required.
\end{proof}

As for boundary regularity, to the best of our knowledge there is no result for solutions of the kind of obstacle problems we are studying. However, if we restrict ourselves to Dirichlet problems (i.e. obstacle problems with $\psi=-\infty$), some results are available. We briefly recall that a point $x_0\in \partial \Omega$ is called ``regular'' if for every function $\theta\in \wup$ continuous in a neighborhood of $x_0$, the unique solution to the relative Dirichlet problem is continuous in $x_0$, and that a necessary and sufficient condition for $x_0$ to be regular is the famous Wiener criterion (which has a local nature). For our purposes, it is enough to use some simpler sufficient conditions for regularity, so we just cite the following corollary of the Wiener criterion:
\begin{teo}[\cite{GZ}, Theorem 2.5]\label{teobjorn}
Let $\Omega$ be a domain, and suppose that $x_0\in\partial\Omega$ has a neighborhood where $\partial \Omega$ is Lipschitz, then $x_0$ is regular for the Dirichlet problem.
\end{teo}
For a more specific discussion of the subject, we refer the reader to \cite{GZ}. We mention that Dirichlet and obstacle problems have been studied also in metric space setting, and boundary regularity theorems with the Wiener criterion have been obtained for example in \cite{BB}, Theorem 7.2.
\begin{oss}
\emph{Note that \cite{GZ} deals only with the case $1<p\leq m$,
but the other cases follows from standard Sobolev embeddings.}
\end{oss}
Using the comparison principle and Proposition \ref{esoluzione}, it is possible to obtain a corollary to this theorem which deals with boundary regularity of some particular obstacle problems.
\begin{cor}\label{obs_bc}
 Consider the obstacle problem $\K_{\psi,\theta}$ on $\Omega$, and suppose that $\Omega$ has Lipschitz boundary and both $\theta$ and $\psi$ are continuous up to the boundary. Then the solution $w$ to $\K_{\psi,\theta}$ is continuous up to the boundary (for convenience we denote $w$ the continuous representative of the solution).
\end{cor}
\begin{proof}
 If we want $\K_{\psi,\theta}$ to be nonempty, it is necessary to assume $\psi(x_0)\leq \theta(x_0)$ for all $x_0\in \partial \Omega$.

Let $\tilde \theta$ be the unique solution to the Dirichlet problem relative to $\theta$ on $\Omega$. Then theorem \ref{teobjorn} guarantees that $\tilde \theta \in C^0(\overline\Omega)$ and the comparison principle allow us to conclude that $w(x)\geq\tilde\theta(x)$ everywhere in $\Omega$.

Suppose first that $\psi(x_0)<\theta(x_0)$, then in a neighborhood $U$ of $x_0$ ($U\subset \Omega$) $w(x)\geq \tilde \theta(x)>\psi(x)$. By Proposition \ref{esoluzione}, $L_{\oF}w=0$ on $U$, and so by Theorem \ref{teobjorn} $w$ is continuous in $x_0$.

If $\psi(x_0)=\theta(x_0)$, consider $w_\epsilon$ the solutions to the obstacle problem $\K_{\tilde\theta+\epsilon,\psi}$. By the same argument as above we have that $w_\epsilon$ are all continuous at $x_0$, and by the comparison principle $w(x)\leq w_\epsilon(x)$ for every $x\in \Omega$ (recall that both functions are continuous in $\Omega$). So we have on one hand:
\begin{gather*}
 \liminf_{x\to x_0} w(x)\geq \liminf_{x\to x_0} \psi(x)=\psi(x_0)=\theta(x_0)
\end{gather*}
and on the other:
\begin{gather*}
 \limsup_{x\to x_0} w(x)\leq \limsup_{x\to x_0} w_{\epsilon}(x)=\theta(x_0)+\epsilon
\end{gather*}
this proves that $w$ is continuous in $x_0$ with value $\theta(x_0)$.
\end{proof}

Finally, we present some results on convergence of supersolutions and their approximation with regular ones.
\begin{prop}\label{convergence}
Let $w_j$ be a sequence of supersolutions on some
open set $\Omega$. 
Suppose that either $w_j\uparrow w$ or $w_j
\downarrow w$ pointwise monotonically, for some locally bounded
$w$. Then, $w$ is a supersolution and there exists a subsequence of $\{w_j\}$ that converges locally strongly in $W^{1,p}$ to $w$ on each compact subset of $\Omega$. Furthermore, if $\{u_j\}$ is a
sequence of solutions of $L_{\oF}u_j=0$ which are locally uniformly bounded in $L^\infty$ and pointwise convergent to
some $u$, then $u$ solves $L_{\oF}u=0$ and, up to choosing a subsequence, $\{u_j\}$ converges to $u$ locally strongly on each compact subset of $\Omega$. 
\end{prop}
\begin{proof}
Suppose that $w_j\uparrow w$. Up to changing the representative in the Sobolev class, by Theorem \ref{regularity} we can assume that $w_j$ is lower semicontinuous. Hence, it has minimum on compact subsets of $\Omega$. Since $w$ is locally bounded and the convergence is monotone up to a set of zero measure, the sequence $\{w_j\}$ turns out to be locally bounded in the $L^\infty$-norm. The elliptic estimate in Proposition
\ref{caccio} ensures that $\{w_j\}$ is locally bounded in
$\wup$. Fix a smooth exhaustion $\{\Omega_n\}$ of $\Omega$. For
each $j$, up to passing to a subsequence, $w_j \rightharpoonup z_n$ weakly in
$W^{1,p}(\Omega_n)$ and strongly in $L^p(\Omega_n)$. By Riesz
theorem, $z_j=w$ for every $j$, hence $w\in W^{1,p}_\loc(\Omega)$.
With a Cantor argument, we can select a sequence, still called
$w_j$, such that $w_j$ converges to $w$ both weakly in
$W^{1,p}(\Omega_n)$ and strongly in $L^p(\Omega_n)$ for every
fixed $n$. To prove that $w$ is a supersolution, fix $0\le\eta \in
C^\infty_c(\Omega)$, and choose a smooth relatively compact open
set $\Omega_0\Subset \Omega$ that contains the support of $\eta$.
Define $M \doteq  \max_j \norm{w_j}_{W^{1,p}(\Omega_0)}<+\infty$. Since
$w_j$ is a supersolution and $w\ge w_j$ for every $j$,
$$
<\oF(w_j), \eta(w-w_j)> \ \ge 0.
$$
Using \eqref{A1} we can rewrite the above inequality as follows:
\begin{equation}\label{primostep}
\int \ps{A(\nabla w_j)}{\eta(\nabla w-\nabla w_j)} \ge -\int
\Big[B(x, w_j)+ \ps{A(\nabla w_j)}{\nabla \eta}\Big](w-w_j).
\end{equation}
Using \eqref{B1}, \eqref{A2} and suitable H\"older inequalities,
the RHS can be bounded from below with the following quantity
\begin{equation}
\begin{array}{l}
\disp -b_1\norm{\eta}_{L^\infty(\Omega)}\int_{\Omega_0}(w-w_j) -
b_2\norm{\eta}_{L^\infty(\Omega)}
\int_{\Omega_0}|w_j|^{p-1}|w-w_j| \\[0.4cm]
\disp  - a_2\norm{\nabla\eta}_{L^\infty(\Omega)}
\int_{\Omega_0}|\nabla
w_j|^{p-1}|w-w_j| \\[0.4cm]
\ge -\norm{\eta}_{C^1(\Omega)}\Big[ b_1|\Omega_0|^{\frac{p-1}{p}}
- b_2 M^{p-1} - a_2M^{p-1}\Big]\norm{w-w_j}_{L^p(\Omega_0)}
\rightarrow 0
\end{array}
\end{equation}
as $j\ra +\infty$. Combining with \eqref{primostep} and the fact
that $w_j \rightharpoonup w$ weakly on $W^{1,p}(\Omega_0)$, by
assumption \eqref{M} the following inequality holds true:
\begin{equation}\label{intazero}
0 \le \int \eta \ps{A(\nabla w)-A(\nabla w_j)}{\nabla w-\nabla
w_j} \le o(1) \qquad \text{as } \ j\ra +\infty.
\end{equation}
By a lemma due to F. Browder (see \cite{browder}, p.13 Lemma 3),
the combination of assumptions $w_j \rightharpoonup w$ both
locally weakly in $W^{1,p}$ and locally strongly in $L^p$, and
\eqref{intazero} for every $0\le \eta\in C^\infty_c(\Omega)$,
implies that $w_j \rightarrow w$ locally strongly in $W^{1,p}$.
Since the operator $\oF$ is weakly continuous, as shown in the
proof of Theorem \ref{existostacolo}, this implies that
$$
0 \ \le \  < \oF(w_j), \eta > \ \longrightarrow \ < \oF(w),\eta >,
$$
hence $L_{\oF}w \le 0$, as required.\\
The case $w_j \downarrow w$ is simpler. By the elliptic
estimate, $w\in W^{1,p}_\loc(\Omega)$, being locally bounded by assumption. Let $\{\Omega_n\}$ be a smooth exhaustion of $\Omega$, and let
$u_n$ be a solution of the obstacle problem relative to $\Omega_n$
with obstacle and boundary value $w$. Then, by \eqref{corsuper}
$w\le u_n \le w_j|_{\Omega_n}$, and letting $j\ra +\infty$ we
deduce that $w=u_n$ is a supersolution on $\Omega_n$, being a
solution of an
obstacle problem.\\
The proof of the last part of the Proposition follows exactly the
same lines as the case $w_j\uparrow w$ done before. Indeed, by the uniform local boundedness, the elliptic estimate gives $\{u_j\} \subset W^{1,p}_\loc(\Omega)$. Furthermore, in
definition $<\oF(u_j), \phi> \ =0$ we can still use as test
function $\phi=\eta(u-u_j)$, since no sign of $\phi$ is
required.
\end{proof}
A couple of corollaries follow from this theorem. It is in fact easy to see that we can relax the assumption of local boundedness on $w$ if we assume a priori $w\in W^{1,p}_\loc(\Omega)$, and moreover with a simple trick we can prove that also local uniform convergence preserves the supersolution property, as in \cite{HKM}, Theorem 3.78.
\begin{cor}
 Let $w_j$ be a sequence of supersolutions locally uniformly converging to $w$, then $w$ is a supersolution.
\end{cor}
\begin{proof}
 The trick is to transform local uniform convergence into monotone convergence. Fix any relatively compact $\Omega_0 \Subset \Omega$ and a subsequence of $w_j$ (denoted for convenience by the same symbol) with $\norm{w_j-w}_{L^\infty(\Omega_0)} \leq 2^{-j}$. The modified sequence of supersolutions $\tilde w_j \doteq w_j + \frac 3 2 \sum_{k=j}^\infty 2^{-k}= w_j+3 \times 2^{-j}$ is easily seen to be a monotonically decreasing sequence on $\Omega_0$, and thus its limit, still $w$ by construction, is a supersolution on any $\Omega_0$ by the previous proposition. The conclusion follows from the arbitrariness of $\Omega_0$.
\end{proof}

Now we prove that with continuous supersolutions we can approximate every supersolution.


\begin{prop}\label{approconti}
 For every supersolution $w\in W^{1,p}_\loc(\Omega)$, there exists a sequence $w_n$ of continuous supersolutions which converge monotonically from below and in $W^{1,p}_\loc(\Omega)$ to $w$. The same statement is true for subsolutions with monotone convergence from above.
\end{prop}
\begin{proof}
Since every $w$ has a lower-semicontinuous representative, it can be assumed to be locally bounded from below, and since $w^{(m)}=\min\{w,m\}$ is a supersolution (for $m\geq 0$) and converges monotonically to $w$ as $m$ goes to infinity, we can assume without loss of generality that $w$ is also bounded above.

Let $\Omega_n$ be a locally finite relatively compact open covering on $\Omega$. Since $w$ is lower semicontinuous it is possible to find a sequence $\phi_m$ of smooth function converging monotonically from below to $w$ (see \cite{HKM}, Section 3.71 p. 75). Let $w_{m}^{(n)}$ be the solution to the obstacle problem $\K_{w,\phi_m}$ on $\Omega_n$. and define $\bar w_m \doteq \min_{n}\{w_m^{(n)}\}$. Thanks to the local finiteness of the covering $\Omega_n$, $\bar w_m$ is a continuous supersolution, being locally the minimum of a finite family of continuous functions. Monotonicity of the convergence is an easy consequence of the comparison principle for obstacle problems, i.e. Proposition \ref{compaosta}. To prove convergence in the local $W^{1,p}$ sense, the steps are pretty much the same as for Proposition \ref{convergence}, and the statement for subsolutions follows from the usual trick.
\end{proof}

\begin{rem}\label{rem_holder}
\rm{With similar arguments and up to some minor technical difficulties, one could strenghten the previous proposition and prove that every supersolution can be approximated by locally H\"older continuous supersolutions.}
\end{rem}

\section{Proof of Theorem \ref{mainteo}}
\begin{teo}
Let $M$ be a Riemannian manifold, and let $A,B$ satisfy the set of
assumptions $\Se$. Define $\oA,\oB,\oF$ as in \eqref{definABF},
and $L_{\oA},L_{\oF}$ accordingly. Then, the following properties
are equivalent:
\begin{itemize}
\item[$(1)$] $(L)$ for $\hol_\loc$ functions,
\item[$(2)$] $(L)$ for $L^\infty$ functions,
\item[$(3)$] $(K)$.
\end{itemize}
\end{teo}
\begin{proof}
$(2)\Rightarrow (1)$ is obvious. To prove that $(1) \Rightarrow
(2)$, we follow the arguments in \cite{PRS2}, Lemma 1.5. Assume by
contradiction that there exists $0\le u\in L^\infty(M)\cap
W^{1,p}_\loc(M)$, $u \not \equiv 0$ such that $L_{\oF}u \ge 0$. We
distinguish two cases.
\begin{itemize}
\item[-] Suppose first that $B(x,u)u$ is not identically zero in
the Sobolev sense. Let $u_2 > u^\star$ be a constant. By
\eqref{B3}, $L_{\oF}u_2 \le 0$. By the subsolution-supersolution
method and the regularity Theorem \ref{regularity}, there exists $w\in \hol_\loc(M)$ such that $u\le w\le u_2$ and
$L_{\oF}w=0$. Since, by \eqref{B2}, \eqref{B3} and $u\le w$,
$B(x,w)w$ is not identically zero, then $w$ is non-constant,
contradicting property $(1)$.
\item[-] Suppose that $B(x,u)u =0$ a.e. on $M$. Since $u$ is
non-constant, we can choose a positive constant $c$ such that both
$\{u-c>0\}$ and $\{u-c<0\}$ have positive measure. By \eqref{B2},
$L_{\oF}(u-c) \ge 0$, hence by Proposition \ref{minsupmaxsub} the
function $v=(u-c)_+ = \max\{u-c,0\}$ is a non-zero subsolution.
Denoting with $\chi_{\{u<c\}}$ the indicatrix of $\{u<c\}$, we can
say that $L_{\oF}v \ge 0 = \chi_{\{u<c\}}v^{p-1}$. Choose any
constant $u_2
> v^\star$. Then, clearly $L_{\oF}u_2 \le \chi_{\{u<c\}}u_2^{p-1}$. Since the potential
$$
\widetilde{B}(x,t)\doteq  B(x,t) + \chi_{\{u<c\}}(x)|t|^{p-2}t
$$
is still a Caratheodory function satisfying the assumptions in
$\Se$, by Theorem \ref{subsuper} there exists a function $w$ such
that $v\le w \le u_2$ and $L_{\oF}w = \chi_{\{u<c\}}w^{p-1}$. By
Theorem \ref{regularity}, $(ii)$ $w$ is locally H\"older
continuous and, since $\{u<c\}$ has positive measure, $w$ is non-constant, contradicting $(1)$.
\end{itemize}
To prove the implication $(3)\Rightarrow (1)$, we follow a
standard argument in potential theory, see for example
\cite{PRS2}, Proposition 1.6. Let $u \in \hol_\loc(M)\cap
W^{1,p}_\loc(M)$ be a non-constant, non-negative, bounded solution
of $L_{\oF}u \ge 0$. We claim that, by the
strong maximum principle, $u<u^\star$ on $M$. Indeed, let
$\widetilde{\mathcal{A}}$ be the operator associated with the
choice $\widetilde{A}(X) \doteq -A(-X)$. Then, since $\widetilde{A}$ satisfies all the
assumptions in $\Se$, it is easy to show that
$L_{\widetilde{\mathcal{A}}}(u^\star -u)\le 0$ on $M$. Hence, by
the Harnack inequality $u^\star -u>0$ on $M$, as desired.\\
Let $K\Subset M$ be a compact set. Consider $\eta$ such that
$0<\eta<u^{\star}$ and define the open set $\Omega_\eta \doteq
u^{-1}\{(\eta,+\infty)\}$. From $u<u^\star$ on $M$, we can choose
$\eta$ close enough to $u^\star$ so that $K\cap
\Omega_\eta=\emptyset$. Let $x_0$ be a point such that
$u(x_0)>\frac{u^\star+\eta}2$, Let $\Omega$ be such that $x_0
\in \Omega$, and choose a \Ka potential relative to the triple
$(K,\Omega, (u^\star-\eta)/2)$. Now, consider the open set $V$
defined as the connected component containing $x_0$ of the open
set
\begin{gather*}
\tilde V \doteq \{x\in \Omega_\eta \ \vert \ u(x)>\eta+ w(x) \}
\end{gather*}
Since $u$ is bounded and $w$ is an exhaustion, $V$ is relatively
compact in $M$ and $u(x)=\eta+w(x)$ on $\partial V$. Since, by
\eqref{B2}, $L_{\oF}(\eta+w)\le 0$, and $L_{\oF}u\ge 0$, this
contradicts the comparison Theorem \ref{teo_comp}.\\
We are left to the implication $(2)\Rightarrow (3)$. Fix a triple
$(K,\Omega,\eps)$, and a smooth exhaustion $\{\Omega_j\}$ of $M$
with $\Omega\Subset \Omega_1$. By the existence Theorem
\ref{existostacolo} with obstacle $\psi = -\infty$, there
exists a unique solution $h_j$ of
$$
\left\{ \begin{array}{l} L_{\oF}h_j =0 \qquad \text{on }
\Omega_j \backslash K \\[0.2cm]
h_j = 0 \quad \text{on } \partial K, \quad h_j = 1 \quad \text{on
}
\partial \Omega_j,
\end{array}
\right.
$$
and $0\le h_j\le 1$ by the comparison Theorem \ref{teo_comp}, with $h_j$ continuous up to $\partial \ton{\Omega_j \setminus K}$ thanks to Theorem \ref{teobjorn}.
Extend $h_j$ by setting $h_j=0$ on $K$ with $h_j=1$ on $M\backslash \Omega_j$. Again
by comparison, $\{h_j\}$ is a decreasing sequence which, by
Proposition \ref{convergence}, converges pointwise on $M$ to a solution
$$
h \in \cap W^{1,p}_\loc(M) \quad \text{of}\quad L_{\oF}h =0 \ \text{on } M\backslash K.
$$
Since $0 \le h \le h_j$ for every $j$, and since $h_j=0$ on $\partial K$, using Corollary \ref{obs_bc} with $\psi= -\infty$ we deduce that $h \in C^0(M)$ and $h=0$ on $K$. We claim that $h= 0$. Indeed, by Lemma \ref{pasting} $u= \max\{h,0\}$ is a
non-negative, bounded solution of $L_{\oF}u\ge 0$
on $M$. By $(1)$, $u$ has to be constant, hence the only possibility is $h = 0$.\\
Now we are going to build by induction an increasing sequence of
continuous functions $\{w_n\}$, $w_0 = 0$, such that:
\begin{enumerate}
 \item[(a)] $w_n|_{K}=0$, $w_n$ are continuous on $M$ and $L_{\oF}w_n \le 0$ on $M\backslash K$,
 \item[(b)] for every $n$, $w_n\leq n$ on all of $M$ and $w_n=n$ in a large enough neighborhood of infinity denoted by $M\backslash C_n$,
 \item[(c)] $\norm{w_n}_{L^\infty(\Omega_n)} \le \norm{w_{n-1}}_{L^\infty(\Omega_n)}+ \frac{\eps}{2^n}$.
\end{enumerate}
Once this is done, by $(c)$ the increasing sequence $\{w_n\}$ is
locally uniformly convergent to a continuous exhaustion which, by
Proposition \ref{convergence}, solves $L_{\oF} w\le 0$ on $M\setminus K$.
Furthermore,
$$
\norm{w}_{L^\infty(\Omega)} \le \sum_{n=1}^{+\infty}
\frac{\eps}{2^n} \le \eps.
$$
Note that $w\in W^{1,p}_{\loc}(M\setminus K)\cap C^0(M)$ with $w=0$ on $K$, so we can conclude immediately that $w\in W^{1,p}_\loc (M)$ and hence $w$ is the desired \Ka potential relative to $(K,\Omega,\eps)$.\\
We start the induction by setting $w_1 \doteq h_j$, for $j$ large enough
in order for property (c) to hold. Define $C_1$ in order to fix
property $(b)$. Suppose now that we have constructed $w_n$. For
notational convenience, write $\bar w=w_n$. Consider the sequence
of obstacle problems $\K_{\bar w + h_j}$ defined on $\Omega_{j+1}\backslash K$ and let $s_j$ be their solution. By Theorem \ref{continuosta} and Corollary \ref{obs_bc} we know that $s_j$ is continuous up to the boundary of its domain. Take for convenience $j$ large enough such that $C_1\subset \Omega_j$. Note
that $s_j|_{\partial K}=0$ and since the constant function $n+1$
is a supersolution, by comparison $s_j\leq n+1$ and
$s_j|_{\Omega_{j+1}\setminus \Omega_j}=n+1$. So we can extend
$s_j$ to a function defined on all of $M$ by setting it equal to
$0$ on $K$ and equal to $n+1$ on $M\backslash \Omega_{j+1}$, and
in this fashion, by Lemma \ref{pasting} $L_{\oF}s_j\le
0$ on $M\setminus K$. By Corollary \ref{corsuper}, $\{s_j\}$ is
decreasing, and so it has a pointwise limit $\bar s$ which is still a
supersolution on $M\setminus K$ by Proposition \ref{convergence}. By Theorem \ref{regularity}, $i)$ the function $\bar s$ admits a lower semicontinuous representative. We are going to
prove that $\bar s = \bar w$. First, we show that $\bar s\leq n$
everywhere. Suppose by contraddiction that this is false. Then,
since $h_j$ converges locally uniformly to zero, on the open set
$A \doteq \bar s^{-1}\{(n,\infty)\}$ the inequality $s_j>\bar w
+h_j$ is locally eventually true, so that $s_j$ is locally
eventually a solution of $L_{\oF}s_j=0$ by Proposition
\ref{esoluzione}, and so $L_{\oF}\bar s = 0$ on $A$ by Proposition
\ref{convergence}. We need to apply the Pasting Lemma \ref{pasting} to the subsolution $\bar s-n$ (defined on $A$) and the zero function. In order to do so, we shall verify that $\max\{\bar s -n,0\} \in X^p_0(A)$, where $X^p_0(A)$ is defined as in \eqref{definXp0}. This requires some care, since $\bar s$ is not a-priori continuous up to $\partial A$. By Proposition \ref{approconti}, we can choose a sequence of continuous supersolutions $\{\sigma_i\} \subset W^{1,p}_\loc(M\backslash K) \cap C^0(M\backslash \overline{K})$ that converges to $\bar s$ both pointwise monotonically and in $W^{1,p}$ on compacta of $M\backslash \overline{K}$. Since $0 \le \bar s \le s_j$ for every $j$, and the sequence $\{s_j\}$ is decreasing, it follows that $\bar s$ is continuous on $\partial K$ with zero boundary value. Therefore, $A$ has positive distance from $\partial K$, and thus $\sigma_i$ converges to $\bar s$ in $W^{1,p}_\loc(\overline{A})$. Since $\bar s$ is lower semicontinuous, $\bar s \le n$ on $\partial A$, so that $\sigma_i \le n$ on $\partial A$ for every $i$. Consequently, the continuous functions $\psi_i = \max\{\sigma_i - n- 1/i,0\}$ converge on compacta of $\overline{A}$ to $\max\{\bar s-n,0\}$, and each $\psi_i$ is zero in a neighbourhood of $\partial A$. This proves the claim that $\max\{\bar s - n,0\} \in X^p_0(A)$. By Lemma \ref{pasting} and assumptions $\Se$, the function
\begin{gather*}
 f \doteq \max\{\bar s -n, 0\}
\end{gather*}
is a non-negative, non-zero bounded solution of $L_{\oF}f \ge 0$.
By $(2)$, $f$ is constant, hence zero; therefore $\bar s \le n$. This proves that $\bar
s=\bar w=n$ on $M\backslash C_n$. As for the remaining set, a
similar argument than the one just used shows that $\bar s$ is a
solution of $L_{\oF}\bar s =0$ on the open, relatively compact set $V \doteq \{\bar s>\bar w\}$, and that $\bar s -w \in W^{1,p}_0(V)$. The comparison principle guarantees that $\bar s \leq \bar w$
everywhere, which is what we needed to prove. Now, since
$s_j\downarrow w$, by Dini's theorem the convergence is locally
uniform and so we can choose $\bar j$ large enough in such a way
that $s_{\bar j}-\bar w < \frac{\eps}{2^n}$ on $\Omega_{n+1}$. Define
$w_{n+1}\doteq s_{\bar j}$, and $C_{n+1}$ in order for $(b)$ to
hold, and the construction is completed.
\end{proof}

\begin{rem}\label{rem_B}
\rm{ As anticipated in Section \ref{sec_deph}, the results of our main theorem are the same if we substitute condition \eqref{B1} with condition \eqref{B1+}:
\begin{gather*}
 |B(x,t)| \le b(t)  \quad \text{instead of} \quad |B(x,t)| \le b_1+b_2 \abs{t}^{p-1}\quad \text{for } t\in \R
\end{gather*}
Although it is not even possible to define the operator $\oB$ if we take $W^{1,p}(\Omega)$ as its domain, this difficulty is easily overcome if we restrict the domain to (essentially) bounded functions, i.e. if we define 
$$
\oB:\wup\cap L^\infty(\Omega)\to \wup^\star.
$$
Now consider that each function used in the proof of the main theorem is either bounded or essentially bounded, so it is quite immediate to see all the existence and comparison theorems proved in Section \ref{sec_tech}, along with all the reasoning and tools used in the proof, still work. Consider for example an obstacle problem $\K_{\theta,\psi}$ such that $\abs \theta \leq C \geq \abs \psi$, and define the operator $\tilde \oB$ relative to the function:
\begin{gather*}
 \tilde B(x,t)= \begin{cases}
                 B(x,t) & \text{ for } \abs t \leq C+1\\
		 b_1(x,C+1) & \text{ for } t\geq C+1\\
                 b_1(x,-(C+1)) & \text{ for } t\leq -(C+1)
                \end{cases}
\end{gather*}
$\tilde \oB$ satisfies evidently condition \eqref{B1}, so it admits a solution to the obstacle problem, which by comparison Theorem \ref{compaosta} is bounded in modulus by $C$, and now it is evident that this function solves also the obstacle problem relative to the original bad-behaved $\oB$.
}

\end{rem}

\section{On the links with the weak maximum
principle and parabolicity: proof of Theorem \ref{LKW}}\label{sec_WMP}
As already explained in the introduction, throughout this section we will restrict ourselves to potentials
$B(x,t)$ of the form $B(x,t)=b(x)f(t)$, where
\begin{equation}
\begin{array}{l}
\disp b, b^{-1} \in L^\infty_\loc(M), \quad b > 0 \text{ a.e. on }
M; \\[0.2cm]
f \in C^0(\R), \quad f(0)=0, \quad f \text{ is non-decreasing on }
\R,
\end{array}
\end{equation}
while we require \eqref{A1}, \eqref{A2} on $A$. 

\begin{oss}
\emph{As in Remark \ref{ossfilapla}, in the case of the operator
$L_\varphi$ in Example \ref{ex2} with $h$ being the metric tensor,
\eqref{A1} and \eqref{A2} can be weakened to \eqref{ipophi} and
\eqref{buonefi}.}
\end{oss}

We begin with the following lemma characterizing $(W)$, whose
proof follows the lines of \cite{PRS3}.

\begin{lemma}\label{lemwmp}
Property $(W)$ for $b^{-1}L_{\oA}$ is equivalent to the following
property, which we call $(P)$:
\begin{quote}
For every $g\in C^0(\R)$, and for every $u\in C^0(M) \cap
W^{1,p}_\loc(M)$ bounded above and satisfying $L_{\oA}u \ge
b(x)g(u)$ on M, it holds $g(u^\star)\le 0$.
\end{quote}
\end{lemma}
\begin{proof}
$(W)\Rightarrow (P)$. From $(W)$ and $L_{\oA}u \ge b(x)g(u)$, for
every $\eta<u^\star$ and $\eps>0$ we can find $0\le \phi \in
C^\infty_c(\Omega_\eta)$ such that
$$
\eps \int b\phi > - <\oA(u), \phi> \ \ge \int g(u)b\phi \ge
\inf_{\Omega_\eta}g(u)\int b\phi
$$
Since $b>0$ a.e. on $M$, we can simplify the integral term to
obtain $\inf_{\Omega_\eta} g(u) \le \eps$. Letting
$\eps\rightarrow 0$ and then $\eta \rightarrow u^\star$, and using
the continuity of $u,g$ we get $g(u^\star)\le 0$, as required. To
prove that $(P)\Rightarrow (W)$, suppose by contradiction that
there exists a bounded above function $u\in C^0\cap W^{1,p}_\loc$,
a value $\eta< u^\star$ and $\eps>0$ such that
$\inf_{\Omega_\eta}b^{-1}L_{\oA}u \ge \eps$. Let $g_\eps(t)$ be a
continuous function on $\R$ such that $g_\eps(t) = \eps$ if $t\ge
u^\star-\eta$, and $g_\eps(t)=0$ for $t\le 0$. Then, by the pasting Lemma \ref{pasting},
$w=\max\{u-\eta,0\}$ satisfies $L_{\oA}w \ge b(x)g_\eps(w)$.
Furthermore, $g_\eps(w^\star)= g_\eps(u^\star-\eta)= \eps$,
contradicting $(P)$.
\end{proof}
Theorem \ref{LKW} is an immediate corollary of the main Theorem \ref{mainteo} and of the following two propositions.
\begin{prop}\label{primaprop}
If $b^{-1}L_{\oA}$ satisfies $(W)$, then $(L)$ holds for every
operator $L_\oF$ of type 1. Conversely, if $(L)$ holds for some
operator $\oF$ of type 1, then $b^{-1}L_{\oA}$ satisfies $(W)$.
\end{prop}
\begin{proof}
Suppose that $(W)$ is met, and let $u \in \hol_\loc \cap
W^{1,p}_\loc$ be a bounded, non-negative solution of $L_{\oF}u\ge
0$. By Lemma \ref{lemwmp}, $f(u^\star)\le 0$. Since $\oF$ is of
type 1, $u^\star \le 0$, that is, $u=0$, as desired. Conversely,
let $\oF$ be an operator of type 1 for which the Liouville
property holds. Suppose by contradiction that $(W)$ is not
satisfied, so that there exists $u \in C^0\cap W^{1,p}_\loc$ such
that $b^{-1}L_{\oA}u \ge \eps$ on some $\Omega_{\eta_0}$. Clearly,
$u$ is non-constant. Since $f(0)=0$, we can choose $\eta \in
(\eta_0,u^\star)$ in such a way that $f(u^\star-\eta)<\eps$.
Hence, by the monotonicity of $f$, the function $u-\eta$ solves
$$
L_{\oA}(u-\eta) \ge b(x)\eps \ge b(x)f(u-\eta) \qquad \text{on } \
\Omega_\eta.
$$
Thanks to the pasting Lemma \ref{pasting}, $w= \max\{u-\eta,0\}$ is a
non-constant, non-negative bounded solution of $L_{\oA}w \ge
b(x)f(w)$, that is, $L_{\oF}w \ge 0$, contradicting the Liouville
property.
\end{proof}
\begin{prop}\label{secondaprop}
If $b^{-1}L_{\oA}$ is parabolic, then $(L)$ holds for every
operator $L_\oF$ of type 2. Conversely, if $(L)$ holds for some
operator $\oF$ of type 2, then $b^{-1}L_{\oA}$ satisfies $(W_\pa)$.
\end{prop}
\begin{proof}
Suppose that $(W_\pa)$ is met. Since each bounded, non-negative $u \in \hol_\loc \cap
W^{1,p}_\loc$ solving $L_{\oF}u\ge
0$ automatically solves $L_{\oA}u\ge 0$, then $u$ is constant by $(W_{\pa})$, which proves $(L)$. Conversely, let $\oF$ be an operator of
type 2 for which the Liouville property holds, and let $[0,T]$ be
the maximal interval in $\R^+_0$ where $f=0$. Suppose by
contradiction that $(W_\pa)$ is not satisfied, so that there
exists a nonconstant $u \in C^0\cap W^{1,p}_\loc$ with
$b^{-1}L_{\oA}u \ge 0$ on $M$. For $\eta$ close enough to
$u^\star$, $u-\eta\le T$ on $M$, hence $w=\max\{u-\eta,0\}$ is a
non-negative, bounded non-constant solution of $L_{\oA}w \ge
0=b(x)f(w)$ on $M$, contradicting the Liouville property for
$\oF$.
\end{proof}
%
%
%

\section{The Evans property}\label{sec_Evans}
We conclude this paper with some comments on the existence of
Evans potentials on model manifolds. It turns out that the function-theoretic properties of these potentials can be used to study the underlying manifold. By a way of example, we quote the papers \cite{VV} and \cite{TS}. In the first one, the authors extend the Kelvin-Nevanlinna-Royden condition and find a Stokes' type theorem for vector fields with integrability condition related to the Evans potential, while in the second article Evans potentials are exploited in order to understand the spaces of harmonic functions with polynomial growth. As a matter of fact, these spaces give a lot of information on the structure at infinity of the manifold. We recall that, only for the standard Laplace-Beltrami operator, it is known that any parabolic Riemannian manifold admits an Evans potential, as proved in \cite{Nakai} or in \cite{SN}, but the technique involved in this proof heavily relies on the linearity of the operator and cannot be easily generalized, even for the $p$-Laplacian. In this respect, see \cite{kilpelainen}.

From the technical point of view, we remark that, for the main
Theorems \ref{mainteo} and \ref{LKW} to hold, no growth control on
$B(x,t)$ in the variable $t$ is required. As we will see, for the
Evans property to hold for $L_{\oF}$ we shall necessarily assume a
precise maximal growth of $B$, otherwise there is no hope to find
any Evans potential. This growth is
described by the so-called Keller-Osserman condition. \par
To begin with, we recall that a model manifold $M_g$ is $\R^m$
endowed with a metric $\di s^2$ which, in polar coordinates
centered at some origin $o$, has the expression $\di s^2 = \di r^2
+ g(r)^2\di \theta^2$, where $\di \theta^2$ is the standard metric
on the unit sphere $\mathbb{S}^{m-1}$ and $g(r)$ satisfies the
following assumptions:
$$
g \in C^\infty(\R^+_0), \quad  g> 0 \text{ on } \R^+, \quad
g'(0)=1, \quad g^{(2k)}(0)=0
$$
for every $k=0,1,2,\ldots$, where $g^{(2k)}$ means the
$(2k)$-derivative of $g$. The last condition ensures that the
metric is smooth at the origin $o$. Note that
$$
\Delta r(x) = (m-1)\frac{g'(r(x))}{g(r(x))}, \quad \vol(\partial
B_r)=g(r)^{m-1}, \quad \vol(B_r) = \int_0^r g(t)^{m-1}\di t.
$$
Consider the operator $L_{\varphi}$ of Example \ref{ex2} with $h$
being the metric tensor. If $u(x)=z(r(x))$ is a radial function, a
straightforward computation gives
\begin{equation}\label{filaradiale}
L_\varphi u =
g^{1-m}\big[g^{m-1}\varphi(|z'|)\mathrm{sgn}(z')\big]'.
\end{equation}
Note that \eqref{doublebound} implies $\varphi(t)\ra +\infty$
as $t\ra +\infty$. Let $B(x,t)=B(t)$ be such that
$$
B \in C^0(\R_0^+), \ B\ge 0 \text{ on } \R^+, \  B(0)=0, \ B
\text{ is non-decreasing on } \R,
$$
and set $B=0$ on $\R^-$. For $c>0$, define the functions
\begin{equation}\label{vez}
\begin{array}{ll}
V_{\pa}(r) = \varphi^{-1}\Big(cg(r)^{1-m}\Big), & \qquad
V_{\stc}(r) = \varphi^{-1}\Big(cg(r)^{1-m}\int_R^rg(t)^{m-1}\di
t\Big) \\[0.3cm]
z_{\pa}(r) = \int_R^r V_\pa(t)\di t, & \qquad z_{\stc}(r) =
\int_R^r V_\stc(t)\di t.
\end{array}
\end{equation}
Note that both $z_\pa$ and $z_\stc$ are increasing on
$[R,+\infty)$. By \eqref{filaradiale}, the functions $u_\pa =
z_\pa \circ r$, $u_\stc=z_\stc \circ r$ are solutions of
$$
L_\varphi u_\pa =0, \qquad L_\varphi u_\stc = c.
$$
Therefore, the following property can be easily verified:
\begin{prop}\label{LKWmodelli}
For the operator $L_{\oF}$ defined by $L_{\oF}u =L_\varphi u -
B(u)$, properties $(K)$ and $(L)$ are equivalent to either
\begin{equation}\label{modstocast}
V_\stc \not \in L^1(+\infty) \quad \text{for every } c>0 \text{
small enough, if } B>0 \text{ on } \R^+,
\end{equation}
or
\begin{equation}\label{modparab}
V_\pa \not \in L^1(+\infty) \quad \text{for every } c>0 \text{
small enough, otherwise.}
\end{equation}
\end{prop}
\begin{proof}
We sketch the proof when $B>0$ on $\R^+$, the other case being
analogous. If $V_\stc \in L^1(+\infty)$, then $u_\stc$ is a
bounded, non-negative solution of $L_\varphi u \ge c$ on
$M\backslash B_R$. Choose $\eta\in (0,u^\star)$ in such a way that
$B(u^\star-\eta) \le c$, and proceed as in the second part of the
proof of Proposition \ref{primaprop} to contradict the Liouville
property of $L_{\oF}$. Conversely, if $V_\stc \not \in
L^1(+\infty)$, then $u_\stc$ is an exhaustion. For every
$\delta>0$, choose $c>0$ small enough that $c\le B(\delta)$. Since
$\varphi(0)=0$, for every $\rho>R$ and $\eps>0$ we can reduce $c$
in such a way that $w_{\eps,\rho}= \delta + u_\stc$ satisfies
$$
w_{\eps,\rho} = \delta \ \text{ on } \partial B_R, \quad
w_{\eps,\rho}\le \delta + \eps \ \text{ on } B_\rho\backslash B_R,
\quad L_\varphi w_{\eps,\rho} = c \le B(\delta) \le
B(w_{\eps,\rho}).
$$
As the reader can check by slightly modifying the argument in the
proof of $(3)\Rightarrow (1)$ of Theorem \ref{mainteo}, the
existence of these modified \Ka potentials for every choice of
$\delta, \eps, \rho$ is enough to conclude the validity of $(L)$,
hence of $(K)$.
\end{proof}
\begin{oss}
\emph{In the case $\varphi(t)=t^{p-1}$ of the $p$-Laplacian,
making the conditions on $V_\stc$ and $V_\pa$ more explicit and
using Theorem \ref{LKW} we deduce that, on model manifolds,
$\Delta_p$ satisfies $(W)$ if and only if
$$
\left(\frac{\vol(B_r)}{\vol(\partial B_r)}\right)^{\frac{1}{p-1}}
\not\in L^1(+\infty),
$$
and $\Delta_p$ is parabolic if and only if
$$
\left(\frac{1}{\vol(\partial B_r)}\right)^{\frac{1}{p-1}} \not\in
L^1(+\infty).
$$
This has been observed, for instance, in \cite{PRS2}, see also the end of \cite{PRS1} and the references therein for a thorough discussion on $\Delta_p$ on model manifolds.}
\end{oss}
We now study the existence of an Evans potential on $M_g$. First, we need to produce radial solutions of $L_\varphi u =B(u)$
which are zero on some fixed sphere $\partial B_R$. To do so, the first step is to solve locally the related Cauchy problem. The
next result is a modification of Proposition A.1 of
\cite{FPR}
%
%
%
%
%
\begin{lemma}
In our assumptions, for every fixed $R > 0$ and $c\in (0,1]$ the
problem
\begin{equation}\label{cauchyloc}
\left\{ \begin{array}{l}
\big[g^{m-1}\varphi(c|z'|)\mathrm{sgn}(z')\big]' =
g^{m-1}B(cz) \qquad \text{on } [R,+\infty) \\[0.2cm]
z(R)=\vartheta\ge 0, \quad z'(R)=\mu> 0
\end{array}\right.
\end{equation}
has a positive, increasing $C^1$ solution $z_c$ defined on a
maximal interval $[R,\rho)$, where $\rho$ may depend on $c$.
Moreover, if $\rho<+\infty$, then $z_c(\rho^-)=+\infty$.
\end{lemma}
\begin{proof}
We sketch the main steps. First, we prove local existence. For
every chosen $r\in (R,R+1)$, denote with $A_\eps$ the $\eps$-ball
centered at the constant function $\vartheta$ in $C^0([R,r],
\|\cdot\|_{L^\infty})$. We look for a fixed point of the Volterra operator
$T_c$ defined by
\begin{equation}\label{operatorT}
T_c(u)(t) = \vartheta + \frac 1 c \int_R^t
\varphi^{-1}\left(\frac{g^{m-1}(R)\varphi(c\mu)}{g^{m-1}(s)} +
\int_R^s \frac{g^{m-1}(\tau)}{g^{m-1}(s)}B(cu(\tau))\di
\tau\right)\di s
\end{equation}
It is simple matter to check the following properties:
\begin{itemize}
\item[$(i)$] If $|r-R|$ is sufficiently small, $T_c(A_\eps)\subset
A_\eps$;
\item[$(ii)$] There exists a constant $C>0$, independent of $r\in
(R,R+1)$, such that $|T_cu(t)-T_cu(s)|\le C|t-s|$ for every $u\in
A_\eps$. By Ascoli-Arzel\`a theorem, $T_c$ is a compact
operator.
\item[$(iii)$] $T_c$ is continuous. To prove this, let
$\{u_j\}\subset A_\eps$ be such that $\|u_j-u\|_{L^\infty}\ra 0$,
and use Lebesgue convergence theorem in the definition of $T_c$ to
show that $T_cu_j\ra T_cu$ pointwise. The convergence is indeed
uniform by $(ii)$.
\end{itemize}
By Schauder theorem (\cite{GT}, Theorem 11.1), $T_c$ has a fixed point $z_c$.
Differentiating $z_c=T_cz_c$ we deduce that $z_c'>0$ on $[R,r]$,
hence $z_c$ is positive and increasing. Therefore, $z_c$ is also a
solution of \eqref{cauchyloc}. This solution can be extended up to
a maximal interval $[R,\rho)$. If by contradiction the
(increasing) solution $z_c$ satisfies $z_c(\rho^-)=z_c^\star
<+\infty$, differentiating $z_c=T_cz_c$ we would argue that
$z_c'(\rho^-)$ exists and is finite. Hence, by local existence
$z_c$ could be extended past $\rho$, a contradiction.
%
\end{proof}
We are going to prove that, if $B(t)$ does not grow too fast and
under a reasonable structure condition on $M_g$, the solution
$z_c$ of \eqref{cauchyloc} is defined on $[R,+\infty)$. To do
this, we first need some definitions. We consider the initial condition
$\vartheta=0$. For convenience, we further require the following
assumptions:
\begin{equation}\label{convenience}
\varphi \in C^1(\R^+), \qquad a_2^{-1}t^{p-1} \le t\varphi'(t) \le
a_1+a_2 t^{p-1} \quad \text{on } \R^+,
\end{equation}
for some positive constants $a_1,a_2$. Define
$$
K_{\mu}(t) = \int_{\mu}^ts\varphi'(s)\di s, \qquad \beta(t) =
\int_{0}^t B(s)\di s.
$$
Note that $\beta(t)$ is non-decreasing on $\R^+$ and that, for
every $\mu \ge 0$, $K_\mu$ is strictly increasing. By
\eqref{convenience}, $K_\mu(+\infty)=+\infty$. We focus our
attention on the condition
\begin{equation}\label{KeOs}\tag{$\urcorner KO$}
\frac{1}{K_\mu^{-1}(\beta(s))} \not \in L^1(+\infty).
\end{equation}
This (or, better, it opposite) is called the Keller-Osserman
condition. Originating, in the quasilinear setting, from works of
J.B. Keller \cite{keller} and R. Osserman \cite{osserman}, it has
been the subject of an increasing interest in the last years. The
interested reader can consult, for instance,
\cite{FPR2}, \cite{MMMR}, \cite{MRS}.
Note that the validity of \eqref{KeOs} is independent of the
choice of $\mu \in [0,1)$, and we can thus refer \eqref{KeOs} to
$K_0=K$. This follows since, by \eqref{convenience}, $K_\mu(t)
\asymp t^p$ as $t\rightarrow +\infty$, where the constant is
independent of $\mu$, and thus $K_\mu^{-1}(s) \asymp s^{1/p}$ as
$s\ra +\infty$, for some constants which are uniform when $\mu \in
[0,1)$. Therefore, \eqref{KeOs} is also equivalent to
\begin{equation}\label{KOsemplice}
\frac{1}{\beta(s)^{1/p}} \not \in L^1(+\infty)
\end{equation}
\begin{lemma}\label{lemma2}
Under the assumptions of the previous proposition and subsequent
discussion, suppose that $g'\ge 0$ on $\R^+$. If
\begin{equation}\label{KO}\tag{$\urcorner KO$}
\frac{1}{K^{-1}(\beta(s))} \not \in L^1(+\infty),
\end{equation}
then, for every choice of $c\in (0,1]$, the solution $z_c$ of
\eqref{cauchyloc} is defined on $[R,+\infty)$.
\end{lemma}
\begin{proof}
From $[g^{m-1}\varphi(cz')]' = g^{m-1}B(cz)$ and $g'\ge 0$ we
deduce that
$$
\varphi'(cz')cz'' \le B(cz), \quad \text{so that} \quad cz'
\varphi'(cz')cz'' \le B(cz)cz' = (\beta(cz))'.
$$
Hence integrating and changing variables we obtain
$$
K_\mu(cz')= \int_{\mu}^{cz'} s\varphi'(s)\di s \le \int_{0}^{cz}
B(s)\di s =\beta(cz).
$$
Applying $K_\mu^{-1}$, $cz'=K_\mu^{-1}(\beta(cz))$. Since $z'>0$,
we can divide the last equality by $K_\mu^{-1}(\beta(cz))$ and
integrate on $[R,t)$ to get, after changing variables,
$$
\int_0^{cz(t)} \frac{\di s}{K_\mu^{-1}(\beta(s))} \le  t-R.
$$
By \eqref{KO}, we deduce that $\rho$ cannot be finite for any
fixed choice of $c$.
\end{proof}
For every $R>0$, we have produced a radial function
$u_c=(cz_c)\circ r$ which solves $L_\varphi u_c = B(u_c)$ on
$M\backslash B_R$ and $u_c=0$ on $B_R$. The next step is to
guarantee that, up to choosing $\mu,c$ appropriately, $u_c$ can be arbitrarily small
on some bigger ball $B_{R_1}$. The basic step is a uniform control of the norm of $z_c$ on $[R,R_1]$ with respect to the variable $c$, up to choosing $\mu=\mu(c)$ appropriately small. This requires a further control on $B(t)$, this time on the whole $\R^+$ and not only in a neighbourhood of $+\infty$.
\begin{lemma}
Under the assumptions of the previous proposition, suppose further that
\begin{equation}\label{omogeneitaB}
B(t) \le b_1 t^{p-1} \qquad \text{on } \R^+.
\end{equation}
Then, for every $R_1>R$ and every $c\in (0,1]$, there exists $\mu>0$ depending $c$ such that the solution
$z_c$ of \eqref{cauchyloc} with $\vartheta=0$ satisfies
\begin{equation}\label{uniformlinfty}
\norm{z_c}_{L^\infty([R,R_1])} \le K,
\end{equation}
for some $K>0$ depending on $R,R_1$, on $a_2$ in \eqref{convenience} and on $b_1$ in \eqref{omogeneitaB} but not on $c$.
\end{lemma}
\begin{proof}
Note that, by \eqref{omogeneitaB}, \eqref{KO} (equivalently, \eqref{KOsemplice}) is satisfied. Hence, $z_c$ is defined on $[R,+\infty)$ for every choice of $\mu, c$. Fix $R_1>R$. Setting $\vartheta=0$ in the expression \eqref{operatorT} of the operator $T_c$, and using the monotonicity of $g$ and $z_c$, we deduce that
$$
\begin{array}{lcl}
u_c(t) & \le & \disp \frac 1c \int_R^t \varphi^{-1}\left(\varphi(c\mu) +
\int_R^s B(cu(\tau))\di
\tau\right)\di s \\[0.4cm]
& \le & \disp \frac 1c \int_R^t \varphi^{-1}\Big(\varphi(c\mu) +
(R_1-R) B(cu_c(s))\Big)\di s.
\end{array}
$$
Differentiating, this gives
$$
\varphi(cu'_c(t)) \le \varphi(c\mu) + (R_1-R)B(cu_c(t)).
$$
Now, from \eqref{convenience} and \eqref{omogeneitaB} we get
\begin{equation}\label{tecnica}
c^{p-1}(u_c')^{p-1} \le a_2 \varphi(c\mu) + a_2(R_1-R)b_2c^{p-1}u_c^{p-1}.
\end{equation}
Choose $\mu$ in such a way that
$$
\varphi(c\mu) \le c^{p-1}, \qquad \text{that is,} \qquad \mu\le \frac 1c \varphi^{-1}(c^{p-1})
$$
Then, dividing \eqref{tecnica} by $c^{p-1}$ and applying the elementary inequality $(x+y)^a \le 2^a(x^a+y^a)$ we obtain the existence of a constant $K=K(R_1,R,a_2,b_2)$ such that
$$
u_c'(t) \le K(1+u_c(t)).
$$
Estimate \eqref{uniformlinfty} follows by applying Gronwall inequality.
\end{proof}
\begin{cor}\label{evansradiali}
Let the assumptions of the last proposition be satisfied. Then, for each triple $(B_R,B_{R_1},\eps)$, there exists a positive, radially increasing solution of $L_\varphi u =B(u)$ on $M_g\backslash B_R$ such that $u=0$ on $\partial B_R$ and $u< \eps$ on $B_{R_1}\backslash B_R$. 
\end{cor}
\begin{proof}
By the previous lemma, for every $c\in (0,1]$ we can choose $\mu=\mu(c)>0$ such that the resulting solution $z_c$ of \eqref{cauchyloc} is uniformly bounded on $[R,R_1]$ by some $K$ independent of $c$. Since, by \eqref{filaradiale}, $u_c = (cz_c) \circ r$ solves $L_\varphi u_c=B(u_c)$, it is enough to choose $c<\eps/K$ to get a desired $u=u_c$ for the triple $(B_R,B_{R_1},\eps)$. 
\end{proof}
To conclude, we shall show that Evans potentials exist for any triple $(K,\Omega,\eps)$, not necessarily given by concentric balls centered at the origin. In order to do so, we use a comparison argument with suitable radial Evans potentials. Consequently, we need to ensure that, for careful choices of $c,\mu$, the radial Evans potentials do not overlap.
\begin{lemma}
Under the assumptions of Lemma \ref{lemma2}, Let $0<R$ be chosen, and let $w$ be a positive, increasing $C^1$ solution of
\begin{equation}
\left\{ \begin{array}{l}
\big[g^{m-1}\varphi(w')\big]' =
g^{m-1}B(w) \qquad \text{on } [R,+\infty) \\[0.2cm]
w(R)=0, \quad w'(R)=w'_R> 0
\end{array}\right.
\end{equation}
Fix $\hat{R}>R$. Then, for every $c>0$, there exists $\mu= \mu(c,R,\hat{R})$ small enough that the solution $z_c$ of \eqref{cauchyloc}, with $R$ replaced by $\hat{R}$, satisfies $cz_c < w$ on $[\hat{R},+\infty)$.
\end{lemma}
\begin{proof}
Let $\mu$ satisfy $g^{m-1}(R)\varphi(w'_R)>g^{m-1}(\hat{R})\varphi(c\mu)$. Suppose by contradiction that $\{cz_c\ge w\}$ is a closed, non-empty set. Let $r>\hat{R}$ be the first point where $cz_c = w$. Then, $cz_c \le w$ on $[\hat{R},r]$, thus $cz_c'(r)\ge w'(r)$. However, from the chain of inequalities
$$
\begin{array}{lcl}
\varphi(w'(r)) & = & \disp \frac{g^{m-1}(R)\varphi(w'_R)}{g^{m-1}(r)} + \int_R^{r} B(w(\tau))\di \tau \\[0.3cm]
& > & \disp \frac{g^{m-1}(\hat{R})\varphi(c\mu)}{g^{m-1}(r)} + \int_{\hat{R}}^r B(cz_c(\tau))\di \tau = \varphi(cz_c'(r)),
\end{array}
$$
and from the strict monotonicity of $\varphi$ we deduce $w'(r)>cz'_c(r)$, a contradiction.
\end{proof}
\begin{cor}\label{notoverlap}
For each $u$ constructed in Corollary \ref{evansradiali}, and for every $R_2>R$, there exists a positive, radially increasing solution $w$ of $L_{\oF}w =0$ on $M_g\backslash B_{R_2}$ such that $w=0$ on $\partial B_{R_2}$ and $w\le u$ on $M\backslash B_{R_2}$. 
\end{cor}
\begin{proof}
It is a straightforward application of the last Lemma.
\end{proof}
We are now ready to state the main result of this section
\begin{teo}
Let $M_g$ be a model with origin $o$ and non-decreasing defining function $g$. Let $\varphi$ satisfies \eqref{convenience} with $a_1=0$, and suppose that $B(t)$ satisfies \eqref{omogeneitaB}. Define $L_{\oF}$ according to $L_{\oF}u=L_\varphi u - B(u)$. Then, properties $(K)$, $(L)$ (for $\hol_\loc$
or $L^\infty$) and $(E)$ restricted to triples $(K,\Omega, \eps)$ with $o\in K$ are equivalent, and also equivalent to either
\begin{equation}\label{modstocast2}
\left(\frac{\vol(B_r)}{\vol(\partial B_r)}\right)^{\frac{1}{p-1}} \not \in L^1(+\infty) \quad \text{if } B>0 \text{ on } \R^+,
\end{equation}
or
\begin{equation}\label{modparab2}
\left(\frac{1}{\vol(\partial B_r)}\right)^{\frac{1}{p-1}} \not \in L^1(+\infty) \quad \text{otherwise.}
\end{equation}
\end{teo}
\begin{proof}
From \eqref{convenience}, assumptions \eqref{modstocast2} and \eqref{modparab2} are equivalent, respectively, to \eqref{modstocast} and \eqref{modparab}.  Therefore, by Proposition \ref{LKWmodelli} and Theorem \ref{mainteo}, the result will be proved once we show that $(L)$ implies $(E)$ restricted to the triples $(K,\Omega,\eps)$ such that $o\in K$. Fix such a triple $(K, \Omega, \eps)$. Since $o\in K$ and $K$ is open, let $R< \rho$ be such that $B_R \Subset K \Subset \Omega \Subset B_\rho$. By making use of Corollary \ref{evansradiali} we can construct a  radially increasing solution $w_2$ of $L_{\oF}w_2 =0$ associated to the triple $(B_R,B_\rho,\eps)$. By $(L)$, $u$ must tend to $+\infty$ as $x$ diverges, for otherwise by the pasting Lemma \ref{pasting} the function $s$ obtained extending $w_2$ with zero on $B_R$ would be a bounded, non-negative, non-constant solution of $L_{\oF}s\ge 0$, contradiction. From Corollary \ref{notoverlap} and the same reasoning, we can produce another exhaustion $w_1$ solving $L_{\oF}w_1=0$ on $M\backslash B_{\rho}$, $w_1=0$ on $\partial B_\rho$ and $w_1\le w_2$ on $M\backslash B_\rho$. Setting $w_1$  equal to zero on $B_\rho$, by the pasting lemma $w_2$ is a global subsolution on $M$ below $w_2$. By the subsolution-supersolution method on $M\backslash K$, there exists a solution $w$ such that $w_1\le w\le w_2$. By construction, $w$ is an exhaustion and $w\le \eps$ on $\Omega \backslash K$. Note that, by Remark \ref{ossregularity}, from \eqref{convenience} with $a_1=0$ we deduce that $w \in C^1(M\backslash K)$. We claim that $w>0$ on $M\backslash K$. To prove the claim we can avail of the strong maximum principle in the form given in \cite{puccirigoliserrin}, Theorem 1.2. Indeed, again from \eqref{convenience} with $a_1=0$ we have (in their notation) 
$$
pa_2^{-1}s^p \le K(s) \le p a_2 s^p \ \text{ on } \R^+, \qquad 0 \le F(s) \le \frac{b_1}{p}s^p \ \text{ on } \R^+,
$$
hence
$$
\frac{1}{K^{-1}(F(s))} \not \in L^1(0^+).
$$
The last expression is a necessary and sufficient condition for the validity strong maximum principle for $C^1$ solutions $u$ of $L_{\oF}u \le 0$. Therefore, $w>0$ on $M\backslash K$ follows since $w$ is not identically zero by contruction. In conclusion, $w$ is an Evans potential relative to $(K,\Omega,\eps)$, as desired.
\end{proof}

\vspace{0.5cm}

\noindent \textbf{Acknowledgements: } We would like to
thank prof. Anders Bjorn for an helpful e-mail discussion, and in
particular for having suggested us the reference to Theorem
\ref{teobjorn}. Furthermore, we thank proff. S. Pigola, M. Rigoli and A.G. Setti for a very careful reading 
of this paper and for their useful comments.

%
%
%
%
%
%
%
%
%

\bibliographystyle{ams_and_hyperref}
\bibliography{bib}

\def\cprime{$'$} \def\cprime{$'$} \def\cprime{$'$} \def\cprime{$'$}
\providecommand{\bysame}{\leavevmode\hbox to3em{\hrulefill}\thinspace}
\providecommand{\MR}{\relax\ifhmode\unskip\space\fi MR }
\providecommand{\MRhref}[2]{%
  \href{http://www.ams.org/mathscinet-getitem?mr=#1}{#2}
}
\providecommand{\href}[2]{#2}
\begin{thebibliography}{10}

\bibitem{antoninimugnaipucci}
Paolo Antonini, Dimitri Mugnai, and Patrizia Pucci, \emph{Quasilinear elliptic
  inequalities on complete {R}iemannian manifolds}, J. Math. Pures Appl. (9)
  \textbf{87} (2007), no.~6, 582--600, \href
  {http://dx.doi.org/10.1016/j.matpur.2007.04.003}
  {\path{doi:10.1016/j.matpur.2007.04.003}}. \MR{2335088 (2008k:58046)}

\bibitem{BB}
Anders Bj{\"o}rn and Jana Bj{\"o}rn, \emph{Boundary regularity for
  {$p$}-harmonic functions and solutions of the obstacle problem on metric
  spaces}, J. Math. Soc. Japan \textbf{58} (2006), no.~4, 1211--1232, \href
  {http://dx.doi.org/10.2969/jmsj/1179759546}
  {\path{doi:10.2969/jmsj/1179759546}}. \MR{2276190 (2008c:35059)}

\bibitem{browder}
Felix~E. Browder, \emph{Existence theorems for nonlinear partial differential
  equations}, Global {A}nalysis ({P}roc. {S}ympos. {P}ure {M}ath., {V}ol.
  {XVI}, {B}erkeley, {C}alif., 1968), Amer. Math. Soc., Providence, R.I., 1970,
  pp.~1--60. \MR{0269962 (42 \#4855)}

\bibitem{FPR}
Roberta Filippucci, Patrizia Pucci, and Marco Rigoli, \emph{Non-existence of
  entire solutions of degenerate elliptic inequalities with weights}, Arch.
  Ration. Mech. Anal. \textbf{188} (2008), no.~1, 155--179, Available from:
  \url{http://dx.doi.org/10.1007/s00205-007-0081-5}, \href
  {http://dx.doi.org/10.1007/s00205-007-0081-5}
  {\path{doi:10.1007/s00205-007-0081-5}}. \MR{2379656 (2009a:35273)}

\bibitem{FPR2}
\bysame, \emph{On weak solutions of nonlinear weighted {$p$}-{L}aplacian
  elliptic inequalities}, Nonlinear Anal. \textbf{70} (2009), no.~8,
  3008--3019, Available from: \url{http://dx.doi.org/10.1016/j.na.2008.12.031},
  \href {http://dx.doi.org/10.1016/j.na.2008.12.031}
  {\path{doi:10.1016/j.na.2008.12.031}}. \MR{2509387 (2010f:35094)}

\bibitem{GZ}
Ronald Gariepy and William~P. Ziemer, \emph{A regularity condition at the
  boundary for solutions of quasilinear elliptic equations}, Arch. Rational
  Mech. Anal. \textbf{67} (1977), no.~1, 25--39, \href
  {http://dx.doi.org/10.1007/BF00280825} {\path{doi:10.1007/BF00280825}}.
  \MR{0492836 (58 \#11898)}

\bibitem{GT}
David Gilbarg and Neil~S. Trudinger, \emph{Elliptic partial differential
  equations of second order}, Classics in Mathematics, Springer-Verlag, Berlin,
  2001, Reprint of the 1998 edition. \MR{1814364 (2001k:35004)}

\bibitem{Grigoryan}
Alexander Grigor{\cprime}yan, \emph{Analytic and geometric background of
  recurrence and non-explosion of the {B}rownian motion on {R}iemannian
  manifolds}, Bull. Amer. Math. Soc. (N.S.) \textbf{36} (1999), no.~2,
  135--249, \href {http://dx.doi.org/10.1090/S0273-0979-99-00776-4}
  {\path{doi:10.1090/S0273-0979-99-00776-4}}. \MR{1659871 (99k:58195)}

\bibitem{HKM}
Juha Heinonen, Tero Kilpel{\"a}inen, and Olli Martio, \emph{Nonlinear potential
  theory of degenerate elliptic equations}, Dover Publications Inc., Mineola,
  NY, 2006, Unabridged republication of the 1993 original. \MR{2305115
  (2008g:31019)}

\bibitem{keller}
J.~B. Keller, \emph{On solutions of {$\Delta u=f(u)$}}, Comm. Pure Appl. Math.
  \textbf{10} (1957), 503--510. \MR{0091407 (19,964c)}

\bibitem{khas}
R.~Z. Khas{\cprime}minski{\u\i}, \emph{Ergodic properties of recurrent
  diffusion processes and stabilization of the solution of the {C}auchy problem
  for parabolic equations}, Teor. Verojatnost. i Primenen. \textbf{5} (1960),
  196--214. \MR{0133871 (24 \#A3695)}

\bibitem{kilpelainen}
Tero Kilpel{\"a}inen, \emph{Singular solutions to {$p$}-{L}aplacian type
  equations}, Ark. Mat. \textbf{37} (1999), no.~2, 275--289, Available from:
  \url{http://dx.doi.org/10.1007/BF02412215}, \href
  {http://dx.doi.org/10.1007/BF02412215} {\path{doi:10.1007/BF02412215}}.
  \MR{1714768 (2000k:31010)}

\bibitem{KS}
David Kinderlehrer and Guido Stampacchia, \emph{An introduction to variational
  inequalities and their applications}, Pure and Applied Mathematics, vol.~88,
  Academic Press Inc. [Harcourt Brace Jovanovich Publishers], New York, 1980.
  \MR{567696 (81g:49013)}

\bibitem{kuratake}
Takeshi Kura, \emph{The weak supersolution-subsolution method for second order
  quasilinear elliptic equations}, Hiroshima Math. J. \textbf{19} (1989),
  no.~1, 1--36, Available from:
  \url{http://projecteuclid.org/getRecord?id=euclid.hmj/1206129479}.
  \MR{1009660 (90g:35057)}

\bibitem{Kuramochi}
Zenjiro Kuramochi, \emph{Mass distributions on the ideal boundaries of abstract
  {R}iemann surfaces. {I}}, Osaka Math. J. \textbf{8} (1956), 119--137.
  \MR{0079638 (18,120f)}

\bibitem{ladyura}
Olga~A. Ladyzhenskaya and Nina~N. Ural{\cprime}tseva, \emph{Linear and
  quasilinear elliptic equations}, Translated from the Russian by Scripta
  Technica, Inc. Translation editor: Leon Ehrenpreis, Academic Press, New York,
  1968. \MR{0244627 (39 \#5941)}

\bibitem{MMMR}
Marco Magliaro, Luciano Mari, Paolo Mastrolia, and Marco Rigoli,
  \emph{Keller-{O}sserman type conditions for differential inequalities with
  gradient terms on the {H}eisenberg group}, J. Diff. Eq. \textbf{250} (2011),
  no.~6, 2643--2670.

\bibitem{ZM}
Jan Mal{\'y} and William~P. Ziemer, \emph{Fine regularity of solutions of
  elliptic partial differential equations}, Mathematical Surveys and
  Monographs, vol.~51, American Mathematical Society, Providence, RI, 1997.
  \MR{1461542 (98h:35080)}

\bibitem{MRS}
Luciano Mari, Marco Rigoli, and Alberto~G. Setti, \emph{Keller-{O}sserman
  conditions for diffusion-type operators on {R}iemannian manifolds}, J. Funct.
  Anal. \textbf{258} (2010), no.~2, 665--712, Available from:
  \url{http://dx.doi.org/10.1016/j.jfa.2009.10.008}, \href
  {http://dx.doi.org/10.1016/j.jfa.2009.10.008}
  {\path{doi:10.1016/j.jfa.2009.10.008}}. \MR{2557951 (2011c:58041)}

\bibitem{Nakai}
Mitsuru Nakai, \emph{On {E}vans potential}, Proc. Japan Acad. \textbf{38}
  (1962), 624--629, \href {http://dx.doi.org/10.3792/pja/1195523234}
  {\path{doi:10.3792/pja/1195523234}}. \MR{0150296 (27 \#297)}

\bibitem{osserman}
Robert Osserman, \emph{On the inequality {$\Delta u\geq f(u)$}}, Pacific J.
  Math. \textbf{7} (1957), 1641--1647. \MR{0098239 (20 \#4701)}

\bibitem{PRS5}
S.~Pigola, M.~Rigoli, and A.~G. Setti, \emph{Maximum principles at infinity on
  {R}iemannian manifolds: an overview}, Mat. Contemp. \textbf{31} (2006),
  81--128, Workshop on Differential Geometry (Portuguese). \MR{2385438
  (2009i:35036)}

\bibitem{PRS4}
Stefano Pigola, Marco Rigoli, and Alberto~G. Setti, \emph{A remark on the
  maximum principle and stochastic completeness}, Proc. Amer. Math. Soc.
  \textbf{131} (2003), no.~4, 1283--1288 (electronic), \href
  {http://dx.doi.org/10.1090/S0002-9939-02-06672-8}
  {\path{doi:10.1090/S0002-9939-02-06672-8}}. \MR{1948121 (2003k:58063)}

\bibitem{PRS3}
\bysame, \emph{Maximum principles on {R}iemannian manifolds and applications},
  Mem. Amer. Math. Soc. \textbf{174} (2005), no.~822, x+99. \MR{2116555
  (2006b:53048)}

\bibitem{PRS2}
\bysame, \emph{Some non-linear function theoretic properties of {R}iemannian
  manifolds}, Rev. Mat. Iberoam. \textbf{22} (2006), no.~3, 801--831, Available
  from: \url{http://projecteuclid.org/getRecord?id=euclid.rmi/1169480031}.
  \MR{2320402 (2008h:31010)}

\bibitem{PRS1}
\bysame, \emph{Aspects of potential theory on manifolds, linear and
  non-linear}, Milan J. Math. \textbf{76} (2008), 229--256, \href
  {http://dx.doi.org/10.1007/s00032-008-0084-1}
  {\path{doi:10.1007/s00032-008-0084-1}}. \MR{2465992 (2009j:31010)}

\bibitem{puccirigoliserrin}
Patrizia Pucci, Marco Rigoli, and James Serrin, \emph{Qualitative properties
  for solutions of singular elliptic inequalities on complete manifolds}, J.
  Differential Equations \textbf{234} (2007), no.~2, 507--543, \href
  {http://dx.doi.org/10.1016/j.jde.2006.11.013}
  {\path{doi:10.1016/j.jde.2006.11.013}}. \MR{2300666 (2008b:35307)}

\bibitem{pucciserrin}
Patrizia Pucci and James Serrin, \emph{The maximum principle}, Progress in
  Nonlinear Differential Equations and their Applications, 73, Birkh\"auser
  Verlag, Basel, 2007. \MR{2356201 (2008m:35001)}

\bibitem{PSZ}
Patrizia Pucci, James Serrin, and Henghui Zou, \emph{A strong maximum principle
  and a compact support principle for singular elliptic inequalities}, J. Math.
  Pures Appl. (9) \textbf{78} (1999), no.~8, 769--789, \href
  {http://dx.doi.org/10.1016/S0021-7824(99)00030-6}
  {\path{doi:10.1016/S0021-7824(99)00030-6}}. \MR{1715341 (2001j:35095)}

\bibitem{SN}
L.~Sario and M.~Nakai, \emph{Classification theory of {R}iemann surfaces}, Die
  Grundlehren der mathematischen Wissenschaften, Band 164, Springer-Verlag, New
  York, 1970. \MR{0264064 (41 \#8660)}

\bibitem{TS}
Chiung-Jue Sung, Luen-Fai Tam, and Jiaping Wang, \emph{Spaces of harmonic
  functions}, J. London Math. Soc. (2) \textbf{61} (2000), no.~3, 789--806,
  \href {http://dx.doi.org/10.1112/S0024610700008759}
  {\path{doi:10.1112/S0024610700008759}}. \MR{1766105 (2001i:31013)}

\bibitem{tolksdorf}
Peter Tolksdorf, \emph{Regularity for a more general class of quasilinear
  elliptic equations}, J. Differential Equations \textbf{51} (1984), no.~1,
  126--150, \href {http://dx.doi.org/10.1016/0022-0396(84)90105-0}
  {\path{doi:10.1016/0022-0396(84)90105-0}}. \MR{727034 (85g:35047)}

\bibitem{valto}
Daniele Valtorta, \emph{Reverse khas'minskii condition}, to appear on Math. Z.,
  \href {http://dx.doi.org/10.1007/s00209-010-0790-6}
  {\path{doi:10.1007/s00209-010-0790-6}}.

\bibitem{VV}
Daniele Valtorta and Giona Veronelli, \emph{Stokes' theorem, volume growth and
  parabolicity}, to appear on Tohoku Math. J. \textbf{63} (2011), no.~3,
  397--412, \href {http://dx.doi.org/10.2748/tmj/1318338948}
  {\path{doi:10.2748/tmj/1318338948}}.

\end{thebibliography}

\end{document}